\documentclass[a4paper,12pt]{article} 
\usepackage{natbib}
\usepackage[utf8]{inputenc}
\usepackage[T1]{fontenc}
\usepackage{lmodern}
\usepackage{amsmath} 
\usepackage{amsthm}
\usepackage{amssymb}
\usepackage{mathtools}
\usepackage{ifthen}
\usepackage{fancyhdr}
\usepackage{graphicx} 
\usepackage[left=1in,top=1in,bottom=1.5in,right=1in]{geometry}

\newcommand*{\keyterm}[1]{\emph{#1}}

\DeclareMathOperator{\re}{Re}			
\DeclareMathOperator{\Span}{span}		
\newcommand{\ran}{\mathcal{R}}						
\DeclareMathOperator*{\dist}{dist}       
\DeclareMathOperator*{\diag}{diag}

\newcommand{\ga}{\alpha}
\newcommand{\gb}{\beta}
\renewcommand{\gg}{\gamma}
\newcommand{\gd}{\delta}
\newcommand{\gl}{\lambda}
\newcommand{\gw}{\omega}
\newcommand{\gs}{\sigma}
\newcommand{\eps}{\varepsilon}

\newcommand*{\C}{{\mathbb{C}}}     
\newcommand*{\R}{{\mathbb{R}}}     
\newcommand*{\Z}{{\mathbb{Z}}}     
\newcommand*{\N}{{\mathbb{N}}}     

\newcommand*{\Lin}{{\mathcal{L}}}   
\newcommand*{\Dom}{{\mathcal{D}}}   

\renewcommand{\ran}{\mathcal{R}} 

\newcommand*{\abs} [1]{\lvert#1\rvert}

\newcommand*{\norm}[1]{\lVert#1\rVert}
\newcommand*{\set} [1]{\{#1\}}
\newcommand*{\iprod}[2]{\langle#1,#2\rangle}    

\newcommand*{\ldelim}[2]{\csname#1l\endcsname#2}   
\newcommand*{\rdelim}[2]{\csname#1r\endcsname#2}   
\newcommand*{\mdelim}[2]{\csname#1m\endcsname#2}   

\newcommand*{\Norm}[1]{\left\lVert#1\right\rVert}
\newcommand*{\Setm}[3][big]{\ldelim{#1}{\{}\,#2\mdelim{#1}{|}#3\,\rdelim{#1}{\}}}

\newcommand*{\Set}[1]{\left\{#1\right\}}

\newcommand*{\Lp}[1][p]{L^{#1}}
\newcommand*{\lp}[1][p]{\ell^{#1}}
\newcommand*{\half}[1][2]{\frac{1}{#1}}
\newcommand*{\inv}{^{-1}}

\newcommand*{\ddb}[2][1]{\ifthenelse{\equal{#1}{1}}{\frac{d}{d#2}}{\frac{d^{#1}}{d#2^{#1}}}}
\newcommand*{\pdb}[2][1]{\ifthenelse{\equal{#1}{1}}{\frac{\partial}{\partial{#2}}}{\frac{\partial^{#1}}{\partial#2^{#1}}}}

\newcommand{\pmat}[1]{\begin{pmatrix}#1\end{pmatrix}}

\newcommand{\eq}[1]{\begin{align*}#1\end{align*}}
\newcommand{\eqn}[1]{\begin{align}#1\end{align}}

\newcommand{\Omi}{\mathcal{O}}

\newtheorem{theorem}{Theorem}
\newtheorem{lemma}[theorem]{Lemma}
\newtheorem{corollary}[theorem]{Corollary}

\theoremstyle{definition}
\newtheorem{definition}[theorem]{Definition}
\newtheorem{example}[theorem]{Example}
\newtheorem{remark}[theorem]{Remark}

\renewcommand{\subsectionmark}[1]{} 
\bibpunct{[}{]}{,}{n}{}{,}

\title{Polynomial Stability of Semigroups Generated by Operator Matrices}
\author{Lassi Paunonen\thanks{Tampere University of Technology, PO.Box 553, 33101 Tampere, Finland, \texttt{lassi.paunonen@tut.fi}}}
\date{~}

\begin{document}

\maketitle
\vspace{-8ex}

\thispagestyle{plain}

\begin{abstract}
In this paper we study the stability properties of strongly continuous semigroups generated by block operator matrices. We consider triangular and full operator matrices whose diagonal operator blocks generate polynomially stable semigroups. As our main results, we present conditions under which also the semigroup generated by the operator matrix is polynomially stable. The theoretic results are applied to deriving conditions for the polynomial stability of a system consisting of a two-dimensional and a one-dimensional damped wave equations.  
\end{abstract}

\smallskip


\section{Introduction}
\label{sec:introduction}

The main topic of this paper is the nonuniform stability of strongly continuous semigroups generated by $2\times 2$ block operator matrices. In particular, we are interested in the asymptotic behaviour of semigroups generated by operators of the form
\eqn{
\label{eq:Aopintro}
A=\pmat{A_1&BC\\0&A_2}, \qquad \mbox{and} \qquad A= \pmat{A_1&B_1C_2\\B_2C_1&A_2}
}
where $A_2: \Dom(A_2)\subset X_2\rightarrow X_2$ and $A_2: \Dom(A_2)\subset X_2\rightarrow X_2$ generate strongly continuous semigroups $T_1(t)$ and $T_2(t)$, respectively, and where $X_1$ and $X_2$ are Hilbert spaces. The rest of the operators are assumed to be bounded. In both of the cases in~\eqref{eq:Aopintro} we denote by $T(t)$ the semigroup generated by $A$ on the Hilbert space $X=X_1\times X_2$.
We concentrate on the situation where the semigroups $T_1(t)$ and $T_2(t)$ are both polynomially stable~\cite{BatEng06,BatDuyPolStab,BorTom10}.
As the main results of this paper, we present conditions under which also the semigroup $T(t)$ generated by $A$ is polynomially stable.

If the semigroups $T_1(t)$ and $T_2(t)$ are exponentially stable, the operator $A$ can be seen as a bounded perturbation of an operator
\eq{
A_0=\pmat{A_1&0\\0&A_2},
}
which generates an exponentially stable semigroup. The perturbation theory for exponentially stable semigroups then states that also the semigroup generated by $A$ is exponentially stable provided that the norms of the operators $BC$, or $B_1C_2$ and $B_2C_1$, are sufficiently small~\cite[Thm. III.1.3]{engelnagel}. In fact, the semigroup generated by the block triangular operator $A$ in~\eqref{eq:Aopintro} is exponentially stable regardless of the size of $\norm{BC}$. However, if the stability of $T_1(t)$ and $T_2(t)$ is not exponential, the situation becomes more complicated, as is illustrated by the following example.

\begin{example}
If $A_1: \Dom(A_1)\subset X_1\rightarrow X_1$ generates a semigroup $T_1(t)$ on $X_1$ and if $\eps>0$, then the block operator matrix
  \eq{
  A = \pmat{A_1 &\eps I\\0&A_1}, \qquad \Dom(A) = \Dom(A_1)\times \Dom(A_1)
  }
  generates a semigroup
  \eq{
  T(t) = \pmat{T_1(t)& \eps t T_1(t)\\0&T_1(t)}
  }
  on $X=X_1\times X_1$.
  In order for this semigroup to be uniformly bounded, it is necessary that
  \eq{
  \sup_{t>0} \,\eps t\norm{T_1(t)}<\infty,
  }
  which implies $\norm{T_1(t)}\rightarrow 0$ as $t\rightarrow \infty$. However, this is only possible if the semigroup $T_1(t)$ is exponentially stable~\cite[Prop. V.1.7]{engelnagel}. This concludes that the semigroup $T(t)$ is unstable whenever the semigroup $T_1(t)$ is not exponentially stable.
\end{example}

In this paper we show that if $T_1(t)$ and $T_2(t)$ are not exponentially stable, then the stability of the semigroup $T(t)$ also depends on other properties of $BC$, $B_1C_2$, and $B_2C_1$ besides their norms. In fact, if $T_1(t)$ and $T_2(t)$ are polynomially stable, 
it is necessary to pose smoothness conditions on these operators in order to guarantee
the stability of $T(t)$.
In particular,
we assume the bounded operators in $A$ satisfy range conditions of the form
\eq{
\ran(B)\subset \Dom( (-A_1)^\gb) \qquad \mbox{and} \qquad \ran(C^\ast) \subset \Dom( (-A_2^\ast)^\gg)
}
for some $\gb,\gg\geq 0$, or
  \eq{
  \ran(B_1)\subset \Dom( (-A_1)^{\gb_1}), \qquad \ran(C_1^\ast ) \subset \Dom( (-A_1^\ast)^{\gg_1})\\
  \ran(B_2)\subset \Dom( (-A_2)^{\gb_2}), \qquad \ran(C_2^\ast ) \subset \Dom( (-A_2^\ast)^{\gg_2})
  }
  for some $\gb_k,\gg_k\geq 0$ and $k=1,2$.
  We will show that the semigroup generated by a triangular $A$ is polynomially stable provided that the exponents $\gb$ and $\gg$ are sufficiently large. In the case of the semigroup generated by a full operator matrix $A$, it is in addition required that the graph norms
  \eq{
  \norm{(-A_1)^{\gb_1}B_1}, \quad
  \norm{(-A_1^\ast)^{\gg_1}C_1^\ast}, \quad 
  \norm{(-A_2)^{\gb_2}B_2}, \quad \mbox{and} \quad 
  \norm{(-A_2^\ast)^{\gg_2}C_2^\ast} 
  }
  are small enough. 

  In addition to our main focus, which is the case where both $T_1(t)$ and $T_2(t)$ are polynomially stable, we separately consider the situations where one of $T_1(t)$ and $T_2(t)$ is exponentially stable and the other is polynomially stable. We show that in such a situation it is possible to completely omit the conditions on the operator $BC$, and relax those on operators $B_2C_1$ and $B_1C_2$ in the stability results. 
In fact, we will see that these conditions agree with the interpretation of exponential stability as the ``limit case'' of polynomial stability with the exponent $\ga=0$.

To the author's knowledge, the polynomial stability of semigroups generated by block operator matrices has not been studied previously in the literature. One known result regarding nonuniform stability of triangular systems states that if one of $T_1(t)$ and $T_2(t)$ is exponentially stable and the other is strongly stable, the semigroup generated a triangular $A$ is also strongly stable, see, for example,~\cite[Lem. 20]{hamalainenpohjolainen10}. The result only applies to triangular systems, and in the corresponding situation for a full operator matrix the stability can in general be destroyed even by operators $B_1C_2$ and $B_2C_1$ with arbitrarily small norms. 
Example~\ref{ex:exppol} in Section~\ref{sec:optex} demonstrates this situation.

The results presented in this paper can be used in studying the asymptotic behaviour of linear partial differential equations. In addition, they also
have applications in the control of infinite-dimensional linear systems. The procedure for stabilizing a linear system using an observer-based dynamic feedback controller requires studying the stability of semigroups generated by block operator matrices, see for example~\cite{CurWei97,hamalainenpohjolainen10,PauPoh12a,PauPoh13b}, and~\cite[Sec. 5.3]{curtainzwart}. If the controlled system is only strongly or polynomially stabilizable, determining the stability of the closed-loop requires results on 
operators of the form~\eqref{eq:Aopintro} where both of $T_1(t)$ and $T_2(t)$ are strongly or polynomially stable.
In particular, since the systems under consideration usually have finite numbers of inputs and outputs, the interconnections corresponding to the operator blocks $BC$, $B_1C_2$, and $B_2C_2$ in~\eqref{eq:Aopintro} are very often finite rank operators.

The operators in~\eqref{eq:Aopintro} can be seen as perturbations of the block-diagonal operator $A_0 = \diag(A_1,A_2)$.
Therefore, the perturbation results in~\cite{Pau12,Pau13a,Pau13c-a} could be used to derive conditions for the stability of the semigroup generated by $A$. 
During the course of this paper we will see that
taking into account the structure of the operator matrices
yields considerably better results.
In particular, the general perturbation results in the above references require that the exponents $\gb$ and $\gg$ are sufficiently large,
and the corresponding graph norms of the perturbing operators are small enough. The results in this paper show that in the case of the triangular block operator matrix, the conditions on the graph norms can be omitted completely. Moreover, for both triangular and full operator matrices the conditions on the exponents $\gb,\gg\geq 0$, and $\gb_1,\gg_1,\gb_2,\gg_2\geq 0$ are weaker than the conditions
achievable by a
direct application of the perturbation results in~\cite{Pau12,Pau13a,Pau13c-a}.

We illustrate the applicability of the theoretic results by studying a system consisting of two damped wave equations --- one two-dimensional and the other one-dimensional.
Both of the wave equations are polynomially stable, and they are coupled in one direction. 
We use our results on triangular systems to derive conditions under which the full connected system is polynomially stable.
In addition, in Section~\ref{sec:optex} we present two shorter examples demonstrating that the conditions on the exponents $\gb,\gg\geq 0$, and $\gb_1,\gg_1,\gb_2,\gg_2\geq 0$ in our main results are, in certain sense, optimal.

The paper is organized as follows. In Section~\ref{sec:mathprelim} we introduce notation and collect some essential results on polynomially stable semigroups. The main results of the paper are presented in Section~\ref{sec:compstab}. The results concerning the stability of semigroups generated by triangular and full systems are proved in Sections~\ref{sec:trisys} and~\ref{sec:fullsys}, respectively. Section~\ref{sec:optex} contains two examples illustrating the optimality of our results. In Section~\ref{sec:waveex} apply the theoretic results to determining the stability of two connected wave equations. Section~\ref{sec:conclusions} contains concluding remarks.

\section{Background on Polynomially Stable Semigroups}
\label{sec:mathprelim}

In this section we introduce the notation used throughout the paper, and review the definition and some of the most important properties of polynomially stable semigroups.
If~$X$ and~$Y$ are Banach spaces and~$A:X\rightarrow Y$ is a linear operator, then we denote by~$\Dom(A)$ and $\ran(A)$ the domain and the range of~$A$, respectively. 
The space of bounded linear operators from~$X$ to~$Y$ is denoted by~$\Lin(X,Y)$. 
If \mbox{$A:\Dom(A)\subset X\rightarrow X$,} then~$\gs(A)$ and~$\rho(A)$ denote the spectrum and the \mbox{resolvent} set of~$A$, respectively. For~$\gl\in\rho(A)$ the resolvent operator is given by \mbox{$R(\gl,A)=(\gl -A)^{-1}$}. 
The inner product on a Hilbert space and the dual pairing on a Banach space are both denoted by $\iprod{\cdot}{\cdot}$.

For a function $f:\R \rightarrow \R$ and for $\ga\geq 0$ we use the notation 
\eq{
f(\gw)=\Omi \left( \abs{\gw}^\ga \right)
}
if there exist constants $M>0$ and $\gw_0\geq 0$ such that
$\abs{f(\gw)}\leq M \abs{\gw}^\ga$
for all $\gw\in \R$ with $\abs{\gw}\geq \gw_0$. 

\begin{definition}
  \label{def:polstab}
  Let $\ga>0$. A semigroup $T(t)$ on a Banach space $X$ generated by $A : \Dom(A) \subset X\rightarrow X$ is \keyterm{polynomially stable with $\ga$}, if $T(t)$ is uniformly bounded, $i\R\subset \rho(A)$, and if there exists $M\geq 1$ such that
  \eq{
  \norm{T(t)A\inv } \leq \frac{M}{t^{1/\ga}}, \qquad \forall t>0.
  }
\end{definition}

For a polynomially stable semigroup $T(t)$ generated by $A$, the operators 
operators $-A$ and $-A^\ast$ are sectorial in the sense of~\cite[Ch. 2]{haasefuncalc} due to the fact that $T(t)$ is uniformly bounded. Therefore, the fractional powers $(-A)^\gb$ and $(-A^\ast)^\gb$ are well-defined for all $\gb\geq 0$.

The following characterizations for polynomial stability of a semigroup on a Hilbert space are essential to the theory presented in this paper. For the proofs of the equivalences, see~\cite[Lem. 2.4]{BatChi13},~\cite[Lem. 2.3, Thm. 2.4]{BorTom10}, and~\cite[Lem. 3.2]{LatShv01}.

\begin{lemma} 
  \label{lem:BorTomresolvent}
  Assume $A$ generates a uniformly bounded semigroup on a Hilbert space $X$, and $i\R\subset\rho(A)$. For fixed $\ga,\gb>0$ the following are equivalent.
  \eq{
  &\mbox{\textup{(a)}} \quad \norm{T_A(t)A\inv} \leq \frac{M}{t^{1/\ga}}, \qquad \forall t>0\nonumber \\[1.5ex]
&\mbox{\textup{(a$'$)}} \quad \norm{T_A(t)(-A)^{-\gb}} \leq \frac{M}{t^{\gb/\ga}}, \qquad \forall t>0\nonumber \\[1.5ex]
  &\mbox{\textup{(b)}} \quad\norm{R(i\gw,A)} = \Omi(\abs{\gw}^\ga) \nonumber\\[2ex]
  &\mbox{\textup{(c)}} \quad\sup_{\re\gl\geq 0} \, \norm{R(\gl,A)(-A)^{-\ga}}<\infty.
  }
\end{lemma}

\begin{lemma} 
  \label{lem:CRBbdd}
Assume $Y$ is a Banach space.
Let $T(t)$ generated by $A$ on a Hilbert space $X$ be polynomially stable with $\ga>0$, and let $B\in \Lin(Y,X)$ and $C\in \Lin(X,Y)$ be such that $\ran(B)\subset \Dom( (-A)^{\gb})$ and $\ran(C^\ast)\subset \Dom( (-A^\ast)^\gg)$
for some $\gb,\gg\geq 0$ satisfying $\gb+\gg\geq \ga$. Then
there exists $M\geq 1$ such that
\eq{
\norm{CR(\gl,A)B}\leq M \norm{(-A)^\gb B} \norm{(-A^\ast)^\gg C^\ast}
}
for all $\gl\in \overline{\C^+}$.
\end{lemma}

\begin{proof}
Since $(-A)^\gb$ has a bounded inverse, $(-A)^\gb$ and $(-A)^\gb B$ are closed operators. Since $\Dom( (-A)^\gb B)=Y$, the Closed Graph Theorem concludes $(-A)^\gb B\in \Lin(Y,X)$. Similarly, we have $(-A^\ast)^\gg C^\ast \in \Lin(Y,X)$.

  Since $(-A^\ast)^\gg C^\ast \in \Lin(Y,X)$, for all $x\in \Dom( (-A)^\gg)$ we have
  \eq{
  \norm{C(-A)^\gg x} 
  = \sup_{\norm{y}=1} \abs{\iprod{C(-A)^\gg x}{y}}
  = \sup_{\norm{y}=1} \abs{\iprod{x}{(-A^\ast)^\gg C^\ast y}}
  \leq  \norm{x}\norm{(-A^\ast)^\gg C^\ast },
  }
  and thus $C(-A)^\gg$ extends to a bounded operator $C_\gg\in \Lin(X,Y)$ with norm $\norm{C_\gg}\leq \norm{(-A^\ast)^\gg C^\ast}$. If we choose 
  \eq{
  M = \norm{(-A)^{\ga-\gb-\gg}} \cdot \sup_{\gl\in \overline{\C^+}}\norm{R(\gl,A)(-A)^{-\ga}},
  }
  then for all $\gl\in \overline{\C^+}$ we have
  \eq{
  \MoveEqLeft\norm{CR(\gl,A)B}
  = \norm{C(-A)^\gg R(\gl,A)(-A)^{-\ga} (-A)^{\ga-\gb-\gg} (-A)^\gb B} \\
  &\leq \norm{C_\gg} \norm{R(\gl,A)(-A)^{-\ga}} \norm{ (-A)^{\ga-\gb-\gg}} \norm{ (-A)^\gb B} 
  \leq M \norm{(-A)^\gb B} \norm{(-A^\ast)^\gg C^\ast}.
  } 
\end{proof}

\begin{remark}
  \label{rem:CBRbddmodremark}
  The proof of Lemma~\ref{lem:CRBbdd} shows that the assumption $\ran(C^\ast)\subset \Dom( (-A^\ast)^\gg)$ could be replaced with the condition that $C(-A)^\gg : \Dom( (-A)^\gg)\subset X\rightarrow Y$ has a bounded extension $C_\gg \in \Lin(X,Y)$. In this version of the result the estimate on the operator $CR(\gl,A)B$ would become
  \eq{
  \norm{CR(\gl,A)B}\leq M \norm{(-A)^\gb B} \norm{C_\gg}.
  }
\end{remark}

\begin{lemma}
  \label{lem:unifbddconds}
  Let $A$ generate a semigroup $T(t)$ on a Hilbert space $X$ and let $\gs(A)\subset \overline{\C^-}$. The semigroup $T(t)$ is uniformly bounded if and only if 
  for all $x,y\in X$ 
  we have
      \eq{
      \sup_{\xi>0}\, \xi \int_{-\infty}^\infty \left( \norm{R(\xi+i\eta,A)x}^2+ \norm{R(\xi+i\eta,A)^\ast y}^2 \right) d\eta <\infty.
      } 
      Moreover, if $\tilde{B}\in \Lin(Y,X)$ where $\dim Y<\infty$, then
  \eq{
  &\sup_{\xi>0} \; \xi\int_{-\infty}^\infty \norm{R(\xi+i\eta,A)\tilde{B}}^2 d\eta <\infty,
  \qquad
  &\sup_{\xi>0} \; \xi\int_{-\infty}^\infty \norm{ R(\xi+i\eta,A)^\ast \tilde{B}}^2 d\eta<\infty.
  } 
\end{lemma}

\begin{proof}
  The proof of the first part can be found in~\cite[Thm. 2]{gomilkounifbdd}.
  If $\dim Y=p<\infty$, we can without loss of generality assume that $Y=\C^p$.
  Let $\set{b_j}_{j=1}^p\subset X$ be such that $\tilde{B}u= \sum_{j=1}^p u_jb_j $ for $u\in \C^p$. A straightforward estimate can be used to show that for any $R\in \Lin(X)$ we have
  $\norm{R \tilde{B}}^2 \leq  \sum_{j=1}^p\, \norm{R b_j}^2$.
  Combining this with the first part of the lemma implies
    \eq{
    &\sup_{\xi>0} \; \xi\int_{-\infty}^\infty \norm{R(\xi+i\eta,A)\tilde{B}}^2 d\eta
    \leq \sum_{j=1}^p \; \sup_{\xi>0} \; \xi\int_{-\infty}^\infty \norm{R(\xi+i\eta,A)b_j}^2 d\eta <\infty\\
    &\sup_{\xi>0} \; \xi\int_{-\infty}^\infty \norm{R(\xi+i\eta,A)^\ast \tilde{B}}^2 d\eta
    \leq \sum_{j=1}^p \; \sup_{\xi>0} \; \xi\int_{-\infty}^\infty \norm{R(\xi+i\eta,A)^\ast b_j}^2 d\eta <\infty.
    }
\end{proof}

\section{Stability of Semigroups Generated by Operator Matrices}
\label{sec:compstab}

In this section we present our main results. The proofs of the theorems are given in Sections~\ref{sec:trisys} and~\ref{sec:fullsys}.
  Throughout the paper we
  assume $T_1(t)$ and $T_2(t)$ are strongly continuous semigroups generated by $A_1: \Dom(A_1) \subset X_1 \rightarrow X_1$ and $A_2: \Dom(A_2) \subset X_2 \rightarrow X_2$, respectively. Most of our results concern the case where both $X_1$ and $X_2$ are Hilbert spaces, and we specifically point out the results that are also valid for Banach spaces.
  Unless otherwise stated, we assume $T_1(t)$ is polynomially stable with $\ga_1>0$, and $T_2(t)$ is polynomially stable with $\ga_2>0$.
  
  Our first main interest is in the stability of the semigroup $T(t)$ generated by
  \eqn{
  \label{eq:Aop}
  A = \pmat{A_1 & BC\\0&A_2},
  \qquad \Dom(A)= \Dom(A_1)\times \Dom(A_2)
  }
  where $B\in \Lin(Y,X_1)$ and $C\in \Lin(X_2,Y)$ for some Banach space $Y$. Since the operator $BC$ is bounded, the semigroup $T(t)$ has the form~\cite[Lem. 3.2.2]{curtainzwart}
  \eq{
  T(t) = \pmat{T_1(t)&S(t)\\0&T_2(t)},
  }
  where $S(t)\in \Lin(X_2,X_1)$ is such that
  \eq{
  S(t)x_2 = \int_0^t T_1(t-s)BCT_2(s)x_2 ds \qquad \forall x_2\in X_2.
  }
We assume the operators $B$ and $C$ satisfy
\eqn{
\label{eq:BCranconds}
\ran(B)\subset \Dom( (-A_1)^\gb) \qquad \mbox{and} \qquad \ran(C^\ast) \subset \Dom( (-A_2^\ast)^\gg)
}
for some $\gb,\gg\geq 0$. As seen in the proof of Lemma~\ref{lem:CRBbdd}, these conditions imply $(-A_1)^\gb B\in \Lin(Y,X_1)$ and $(-A_2^\ast)^\gg C^\ast \in \Lin(Y,X_2)$.

The first two results provide sufficient conditions for the stability of the semigroup $T(t)$ on Hilbert and Banach spaces, respectively. 

\begin{theorem}
  \label{thm:polpoltri}
Assume $X_1$ and $X_2$ are Hilbert spaces.
If $\gb/\ga_1 + \gg/\ga_2> 1$, then the semigroup
generated by $A$ in~\eqref{eq:Aop} is polynomially stable with $\ga = \max \set{\ga_1,\ga_2}$. If $\dim Y<\infty$, then it is sufficient that $\gb/\ga_1 + \gg/\ga_2\geq 1$.
\end{theorem}

\begin{theorem}
  \label{thm:polpoltribanach}
  Assume $X_1$, $X_2$, and $Y$ are Banach spaces.
  If $\gb/\ga_1 + \gg/\ga_2> 1$, then the semigroup
  generated by $A$ in~\eqref{eq:Aop} is strongly stable.
   If $\ga = \ga_1+\ga_2$, then there exists $M\geq 1$ such that
  \eq{
  \norm{T(t)A\inv} \leq M\left( \frac{ \ln t}{t} \right)^{1/\ga} \qquad \forall t>0.
  }
\end{theorem}

If one of the subsystems is exponentially stable, then the requirements on the exponents $\gb$ and $\gg$ can be omitted completely.

\begin{theorem}
  \label{thm:polexptri}
  If $T_1(t)$ is exponentially stable, then the semigroup $T(t)$ generated by $A$ in~\eqref{eq:Aop} is polynomially stable with $\ga=\ga_2$. Similarly, if $T_2(t)$ is exponentially stable, then $T(t)$ is polynomially stable with $\ga=\ga_1$.
\end{theorem}

  The above results are stated for upper triangular systems, but the analogous results are also valid for lower triangular systems. 
Indeed, any lower triangular block operator matrix can be transformed into an upper triangular one with a similarity transformation
\eq{
 \pmat{0&I\\I&0} \pmat{A_1&BC\\0&A_2}\pmat{0&I\\I&0}
= \pmat{A_2&0\\BC&A_1}.
}
Since the stability properties considered in this paper are invariant under similarity transformations, Theorems~\ref{thm:polpoltri},~\ref{thm:polpoltribanach}, and~\ref{thm:polexptri} also provide
conditions for stability of semigroups generated by lower triangular block operator matrices.

The rest of the results in this section concern the stability of the semigroup generated by an operator of the form
\eqn{
\label{eq:Aopfull}
A= \pmat{A_1&B_1C_2\\B_2C_1&A_2},
  \qquad \Dom(A)= \Dom(A_1)\times \Dom(A_2)
}
where 
$B_1\in \Lin(Y_1,X_1)$, $B_2\in \Lin(Y_2,X_2)$, $C_1\in \Lin(X_1,Y_2)$, and $C_2\in \Lin(X_2,Y_1)$ satisfy 
\begin{subequations}
  \label{eq:BCrancondsfull} 
  \eqn{
  \label{eq:BCrancondsfull1} 
  \ran(B_1)\subset \Dom( (-A_1)^{\gb_1}), \quad \ran(C_1^\ast ) \subset \Dom( (-A_1^\ast)^{\gg_1})\\
  \label{eq:BCrancondsfull2} 
  \ran(B_2)\subset \Dom( (-A_2)^{\gb_2}), \quad \ran(C_2^\ast ) \subset \Dom( (-A_2^\ast)^{\gg_2})
  }
\end{subequations}
for some $\gb_1,\gb_2,\gg_1,\gg_2\geq 0$.

The following theorem presents conditions for the polynomial stability of $T(t)$. 

\begin{theorem} 
  \label{thm:polpolfull}
  Assume $X_1$ and $X_2$ are Hilbert spaces
  and let one of the following conditions be satisfied:
  \begin{itemize}
    \item[\textup{(i)}] $\gb_1,\gg_1\geq \ga_1$ and $\gb_2,\gg_2\geq\ga_2$
    \item[\textup{(ii)}] $\dim Y_1<\infty$, $\gb_1+\gg_1\geq \ga_1$, and $\gb_2,\gg_2\geq \ga_2$
    \item[\textup{(iii)}] $\dim Y_2<\infty$, $\gb_1,\gg_1\geq \ga_1$, and $\gb_2+\gg_2\geq \ga_2$
    \item[\textup{(iv)}]  $\dim Y_1<\infty$, and $\dim Y_2<\infty$ and  $\gb_k/\ga_k+\gg_l/\ga_l\geq 1$ for every $k,l\in \set{1,2}$.
  \end{itemize}
There exists $\gd>0$ such that if $B_1,C_1,B_2$, and $C_2$ satisfy~\eqref{eq:BCrancondsfull} and
  \eqn{
  \label{eq:polpolfullnorms}
  \norm{(-A_1)^{\gb_1}B_1} \cdot
  \norm{(-A_1^\ast)^{\gg_1}C_1^\ast} \cdot
  \norm{(-A_2)^{\gb_2}B_2} \cdot
  \norm{(-A_2^\ast)^{\gg_2}C_2^\ast} <\gd,
  }
  then the semigroup generated by $A$ in~\eqref{eq:Aopfull} is polynomially stable with $\ga=\max \set{\ga_1,\ga_2}$.
\end{theorem}

Written out explicitly, the conditions (iv) for the exponents in Theorem~\ref{thm:polpolfull} become
\eq{
\gb_1/\ga_1+\gg_2/\ga_2\geq 1,\\
\gb_2/\ga_2+\gg_1/\ga_1\geq 1,\\
\gb_1+\gg_1\geq \ga_1,\\
\gb_2+\gg_2\geq \ga_2.
}

As already mentioned, an alternative approach to the stability of $T(t)$ would be to write
\eq{
A = \pmat{A_1&0\\0&A_2} + \pmat{B_1&0\\0&B_2}\pmat{0&C_2\\C_1&0} 
=: A_0 +\Delta_B \Delta_C,
}
and apply the perturbation results in~\cite{Pau12,Pau13a,Pau13c-a}. Here $\Delta_B\in \Lin(Y_1\times Y_2,X)$ and $\Delta_C \in \Lin(X,Y_1\times Y_2)$.
Indeed, we have $\ran(\Delta_B) = \ran(B_1)\times \ran(B_2)$, $\ran(\Delta_C^\ast) = \ran(C_1^\ast)\times \ran(C_2^\ast)$, and
\eq{
\norm{(-A_0)^\gb\Delta_B} 
&= \Norm{\pmat{(-A_1)^\gb B_1& 0\\0&(-A_2)^\gb B_2}}
= \max \set{\norm{(-A_1)^\gb B_1},\norm{(-A_2)^\gb B_2}}\\
\norm{(-A_0^\ast)^\gg\Delta_C^\ast} 
&= \Norm{\pmat{0&(-A_1^\ast)^\gg C_1^\ast\\(-A_2^\ast)^\gg C_2^\ast&0}}
= \max \set{\norm{(-A_1^\ast)^\gg C_1^\ast},\norm{(-A_2^\ast)^\gg C_2^\ast}} 
}
if $\gb\geq \min \set{\gb_1,\gb_2}$ and $\gg \geq \min \set{\gg_1,\gg_2}$.
The conditions on the exponents resulting from the direct application of the perturbation results thus become~\cite[Thm. 6]{Pau13c-a}
\begin{itemize}
  \item[(i)] $\gb_1,\gg_1,\gb_2,\gg_2\geq \max \set{\ga_1,\ga_2}$, or
  \item[(ii)] $\min \set{\gb_1,\gb_2}+\min \set{\gg_1,\gg_2}\geq \max \set{\ga_1,\ga_2}$ if $\dim Y_1<\infty$ and $\dim Y_2<\infty$.
\end{itemize}
If one of the above conditions is satisfied, then the semigroup $T(t)$ is polynomially stable whenever the product 
\eq{
 \max \set{\norm{(-A_1)^\gb B_1},\norm{(-A_2)^\gb B_2}}
\cdot \max \set{\norm{(-A_1^\ast)^\gg C_1},\norm{(-A_2^\ast)^\gg C_2^\ast}} 
}
is small enough.
The conditions in Theorem~\ref{thm:polpolfull} have two advantages over this approach: (1) The conditions on the exponents are less strict in the cases where $\ga_1\neq \ga_2$, or where one of $Y_1$ and $Y_2$ is finite-dimensional and (2) making any one of the norms in~\eqref{eq:polpolfullnorms} small can be used to compensate for the size of the other three norms.

If the semigroup $T_1(t)$ is exponentially stable, it is possible to remove the requirements on the exponents $\gb_1$ and $\gg_1$ from the assumptions.

\begin{theorem} 
  \label{thm:polexpfull}
  Assume $T_1(t)$ is exponentially stable and $\gb_2,\gg_2\geq \ga_2$. 
There exists $\gd>0$ such that if $B_2$, and $C_2$ satisfy~\eqref{eq:BCrancondsfull2} and
  \eq{
  \norm{B_1} \cdot
  \norm{C_1} \cdot
  \norm{(-A_2)^{\gb_2}B_2} \cdot
  \norm{(-A_2^\ast)^{\gg_2}C_2^\ast} <\gd,
  }
  then the semigroup generated by $A$ in~\eqref{eq:Aopfull} is polynomially stable with $\ga=\ga_2$.
  If $\dim Y_2<\infty$, it is sufficient that the exponents satisfy $\gb_2+\gg_2\geq \ga_2$.
\end{theorem}

\begin{remark}
  \label{rem:explimit}
  Lemma~\ref{lem:BorTomresolvent} shows that on a Hilbert space the exponential stability of a semigroup can be seen as a ``limit case'' of polynomial stability with exponent $\ga=0$. Indeed, if $\gs(A)\subset \C^-$ and if $T(t)$ is uniformly bounded, then the condition $\norm{R(i\gw,A)}=\Omi(\abs{\gw}^\ga)$ with $\ga=0$ is equivalent to $T(t)$ being exponentially stable due to the characterization by Gearhart, Pr\"{u}ss, and Greiner~\cite[Thm. V.1.11]{engelnagel},~\cite[Thm. 5.1.5]{curtainzwart}.
  Comparing Theorems~\ref{thm:polpolfull} and~\ref{thm:polexpfull} shows that the results on polynomial stability of $T(t)$ agree with this interpretation. In particular, if $T_1(t)$ is exponentially stable, we would then have $\ga_1=0$, and the conditions on $\gb_1$, and $\gg_1$ would be satisfied with the choice $\gb_1=\gg_1=0$. Under the conditions of Theorem~\ref{thm:polpolfull}, the semigroup $T(t)$ would be polynomially stable with $\ga = \max \set{\ga_1,\ga_2}=\ga_2$.
   The conditions in Theorems~\ref{thm:polpoltri} and~\ref{thm:polexptri} are related to each other in a similar way. 
\end{remark}

Applying a similarity transformation
\eq{
 \pmat{0&I\\I&0} \pmat{A_1&B_1C_2\\B_2C_1&A_2}\pmat{0&I\\I&0}
=  \pmat{A_2&B_2C_1\\B_1C_2&A_1},
}
yields the following analogue of Theorem~\ref{thm:polexpfull} concerning the case where $T_2(t)$ is exponentially stable.

\begin{corollary} 
  \label{cor:polexpfullinv}
  Assume $T_2(t)$ is exponentially stable and $\gb_1,\gg_1\geq \ga_1$. 
There exists $\gd>0$ such that if $B_1$, and $C_1$ satisfy~\eqref{eq:BCrancondsfull1} and
  \eq{
  \norm{(-A_1)^{\gb_1}B_1} \cdot
  \norm{(-A_1^\ast)^{\gg_1}C_1^\ast} \cdot
  \norm{B_2} \cdot
  \norm{C_2} <\gd,
  }
  then the semigroup generated by $A$ in~\eqref{eq:Aopfull} is polynomially stable with $\ga=\ga_1$.
  If $\dim Y_1<\infty$, it is sufficient that the exponents satisfy $\gb_1+\gg_1\geq \ga_1$.
\end{corollary}

We begin by presenting the proofs for the results concerning semigroup generated by triangular operator matrices. The results on the stability of semigroup generated by full operator matrices are proved in Section~\ref{sec:fullsys}.

\section[Triangular Operator Matrices]{Semigroups Generated By Triangular Operator Matrices}
\label{sec:trisys}

In this section we present the proofs for Theorems~\ref{thm:polpoltri},~\ref{thm:polpoltribanach}, and~\ref{thm:polexptri} concerning the stability of the semigroup generated by the triangular operator matrix
\eq{
A= \pmat{A_1&BC\\0&A_2}.
}
For a triangular operator matrix, the spectral properties of $A$ are determined by those of the operators $A_1$ and $A_2$.

\begin{lemma} 
  \label{lem:trispec}
  Assume $X_1$, $X_2$, and $Y$ are Banach spaces.
  The spectrum of $A$ satisfies $\gs(A)\subset \C^-$ and
\eqn{
\label{eq:trires}
  R(\gl,A) = \pmat{R(\gl,A_1) & R(\gl,A_1)BC R(\gl,A_2)\\0&R(\gl,A_2)} 
}
for every $\gl\in \overline{\C^+}$.
\end{lemma}

\begin{proof}
  Let $\gl\in \overline{\C^+}$ be arbitrary. Since $\gl \in \rho(A_1)$ and $\gl \in \rho(A_2)$, a direct computation shows that
  $\gl-A$ has a bounded inverse given by the right-hand side of~\eqref{eq:trires}. This immediately implies   $\gl\in\rho(A)$.
\end{proof}

\begin{lemma}
  \label{lem:triunifbdd}
  If $X_1$, $X_2$, and $Y$ are Banach spaces and $\gb/\ga_1+\gg/\ga_2>1$, then the semigroup $T(t)$ is uniformly bounded.  
\end{lemma}

\begin{proof}
  Since the semigroups $T_1(t)$ and $T_2(t)$ are uniformly bounded, the semigroup $T(t)$ is uniformly bounded if (and only if) the operators $S(t)$ are uniformly bounded with respect to $t\geq 0$. Denote $M_1=\sup_{t\geq 0}\norm{T_1(t)}$ and $M_2=\sup_{t\geq 0}\norm{T_2(t)}$. Moreover, let $M_\gb,M_\gg\geq 1$ be such that 
  \eq{
  \norm{T_1(t)(-A_1)^{-\gb}}\leq \frac{M_\gb}{t^{\gb/\ga_1}}, 
  \qquad \mbox{and} \qquad
  \norm{T_2(t)(-A_2)^{-\gg}}\leq \frac{M_\gg}{t^{\gg/\ga_2}}, 
  }
  for all $t>0$.
  Let $x\in X_2$ and $t\geq 2$.
  If we denote $B_\gb = (-A_1)^\gb B\in \Lin(Y,X_1)$ and $C_\gg = \overline{C(-A_2)^\gg}\in \Lin(X_2,Y)$, then for $s\in [1,t-1]$ we have
  \eq{
  \norm{T_1(t-s)BCT_2(s)x}
  &=\norm{T_1(t-s)(-A_1)^{-\gb}(-A_1)^\gb BC(-A_2)^\gg  T_2(s)(-A_2)^{-\gg}x}\\
  &\leq  \norm{T_1(t-s)(-A_1)^{-\gb}}\norm{B_\gb}\norm{C_\gg}  \norm{  T_2(s)(-A_2)^{-\gb}}\norm{x}\\
  &\leq  M_\gb M_\gg\norm{B_\gb}\norm{C_\gg} \norm{x} \cdot \frac{1}{(t-s)^{\gb/\ga_1}}\cdot \frac{1}{s^{\gg/\ga_2}}
  }
   and thus
\eq{
\norm{S(t)x}
&\leq \int_0^t \norm{T_1(t-s)BCT_2(s)x}ds\\
&\leq 
\int_0^1 \norm{T_1(t-s)}\norm{BC}\norm{T_2(s)}\norm{x}ds
+\int_1^{t-1}\norm{T_1(t-s)BCT_2(s)x}ds\\
&\quad +\int_{t-1}^t \norm{T_1(t-s)} \norm{BC} \norm{T_2(s)} \norm{x}ds\\
&\leq 2M_1M_2\norm{BC}\norm{x}
+ M_\gb M_\gg \norm{B_\gb}\norm{C_\gg} \norm{x} \int_1^{t-1} \frac{1}{(t-s)^{\gb/\ga_1}}\cdot \frac{1}{s^{\gg/\ga_2}}ds.
}
Since $x\in X_2$ was arbitrary, we have that $\sup_{t\geq 0}\norm{S(t)}<\infty$ if
\eqn{
\label{eq:triunifbddintcond}
\sup_{t\geq 2} \; \int_1^{t-1} (t-s)^{-\gb/\ga_1} s^{-\gg/\ga_2}ds<\infty.
}
If $\gg/\ga_2>1$, then
\eq{
 \int_1^{t-1} (t-s)^{-\gb/\ga_1} s^{-\gg/\ga_2}ds
 \leq \int_1^{t-1}  s^{-\gg/\ga_2}ds
 \leq \int_1^\infty  s^{-\gg/\ga_2}ds<\infty,
}
and if $\gb/\ga_1>1$, then
\eq{
 \int_1^{t-1} (t-s)^{-\gb/\ga_1} s^{-\gg/\ga_2}ds
 \leq \int_1^{t-1} (t-s)^{-\gb/\ga_1}ds
 = \int_1^{t-1} r^{-\gb/\ga_1}dr
 \leq \int_1^\infty r^{-\gb/\ga_1}dr <\infty.
}
In both of these cases~\eqref{eq:triunifbddintcond} is satisfied. It remains to consider the case where $0<\gb/\ga_1\leq 1$ and $0<\gg/\ga_2 \leq 1$ satisfy $\gb/\ga_1+\gg/\ga_2>1$. Choose $c=\gb/\ga_1+\gg/\ga_2> 1$, $p=c\ga_1/\gb$, and $q=c\ga_2/\gg$. Then
  \eq{
  \frac{1}{p} + \frac{1}{q} = \frac{1}{c} (\gb/\ga_1+\gg/\ga_2) = 1.
  }  
  Since $p>\ga_1/\gb\geq 1$, and $q>\ga_2/\gg\geq 1$, and since $p\cdot \gb/\ga_1=c>1$ and $q\cdot \gg/\ga_2=c>1$, the Hölder inequality with exponents $p$ and $q$ shows that
\eq{
 \MoveEqLeft\int_1^{t-1} (t-s)^{-\gb/\ga_1}\cdot s^{-\gg/\ga_2}ds
 \leq \Bigl( \int_1^{t-1} (t-s)^{-p\cdot\gb/\ga_1} ds \Bigr)^{1/p} \Bigl( \int_1^{t-1} s^{-q\cdot\gg/\ga_2}ds \Bigr)^{1/q}\\
 &= \Bigl( \int_1^{t-1} r^{-c} dr \Bigr)^{1/p} \Bigl( \int_1^{t-1} s^{-c}ds \Bigr)^{1/q}
 \leq \Bigl( \int_1^\infty r^{-c} dr \Bigr)^{1/p} \Bigl( \int_1^\infty s^{-c}ds \Bigr)^{1/q} <\infty.
 }
 This shows that~\eqref{eq:triunifbddintcond} holds also in the case where $0<\gb/\ga_1\leq 1$, $0<\gg/\ga_2 \leq 1$, and $\gb/\ga_1+\gg/\ga_2>1$. This finally concludes that $\sup_{t\geq 0}\norm{S(t)}<\infty$, and thus $T(t)$ is uniformly bounded.
\end{proof}

\begin{lemma}
  \label{lem:BCintprops}
  Assume $X_1$, $X_2$ are Hilbert spaces and $Y_1$ and $Y_2$  are Banach spaces, and that $\tilde{B}\in \Lin(Y_1,X_1)$ and $\tilde{C}\in \Lin(X_2,Y_2)$ satisfy $\ran(\tilde{B})\subset \Dom( (-A_1)^\gb)$ and $\ran(\tilde{C}^\ast)\subset \Dom( (-A_2^\ast)^\gg)$ for some $\gb,\gg\geq 0$.
  If $\gb/\ga_1+ \gg/\ga_2\geq 1$,
  then $\norm{R(i\gw,A_1)\tilde{B}} \norm{\tilde{C}R(i\gw,A_2)} = \Omi(\abs{\gw}^\ga)$ with $\ga = \max \set{\ga_1,\ga_2}$. Moreover, if $\gb\geq \ga_1$, we then have
  \eqn{
  \label{eq:RBCRint}
  \sup_{\xi>0} \; \xi \int_{-\infty}^\infty \norm{R(\xi+i\eta,A_1)\tilde{B}}^2 \norm{\tilde{C}R(\xi+i\eta,A_2)x}^2 d\eta <\infty \qquad \forall x\in X_2.
  }
If $\dim Y_1<\infty$, then~\eqref{eq:RBCRint} is satisfied whenever $\gb/\ga_1+ \gg/\ga_2\geq 1$.
\end{lemma}

\begin{proof}
By Lemma~\ref{lem:BorTomresolvent} we can choose $M_0\geq 1$ such that $\norm{R(\gl,A_1)(-A_1)^{-\ga_1}}\leq M_0$ and $\norm{R(\gl,A_2)(-A_2)^{-\ga_2}}\leq M_0$ for all $\gl\in \overline{\C^+}$.

If $\gb/\ga_1\geq 1$, then for every $\gl\in \overline{\C^+}$ and $x\in X_2$ we can estimate
\eq{
\MoveEqLeft \norm{R(\gl,A_1)\tilde{B}} \norm{\tilde{C}R(\gl,A_2) x} 
= \norm{R(\gl,A_1)(-A_1)^{-\ga_1}(-A_1)^{\ga_1}\tilde{B}} \norm{\tilde{C}R(\gl,A_2) x} \\
&\leq \norm{R(\gl,A_1)(-A_1)^{-\ga_1}}\norm{(-A_1)^{\ga_1}\tilde{B}} \norm{\tilde{C}} \norm{R(\gl,A_2) x} \\
&\leq M_0\norm{(-A_1)^{\ga_1}\tilde{B}} \norm{\tilde{C}} \norm{R(\gl,A_2) x} ,
}
which in particular implies $\norm{R(i\gw,A_1)\tilde{B}} \norm{\tilde{C}R(i\gw,A_2)} = \Omi(\abs{\gw}^\ga)$ due to $\norm{R(i\gw,A_2)}=\Omi(\abs{\gw}^\ga)$. Moreover, we then have
\eq{
\MoveEqLeft\sup_{\xi>0} \; \xi \int_{-\infty}^\infty \norm{R(\xi+i\eta,A_1)\tilde{B}}^2 \norm{\tilde{C}R(\xi+i\eta,A_2)x}^2 d\eta \\
&\leq M_0^2\norm{(-A_1)^{\ga_1}\tilde{B}}^2 \norm{\tilde{C}}^2 \sup_{\xi>0} \; \xi \int_{-\infty}^\infty  \norm{R(\xi+i\eta,A_2) x}^2 d\eta <\infty
}
by Lemma~\ref{lem:unifbddconds}. Since $x\in X_2$ was arbitrary, this concludes that~\eqref{eq:RBCRint} is satisfied.

If $\gg/\ga_2\geq 1$, then for every $\gl\in \overline{\C^+}$
\eq{
\MoveEqLeft \norm{R(\gl,A_1)\tilde{B}} \norm{\tilde{C}R(\gl,A_2)x} 
= \norm{R(\gl,A_1)\tilde{B}} \norm{\tilde{C}(-A_2)^{\ga_2}(-A_2)^{-\ga_2}R(\gl,A_2)x} \\
&\leq \norm{R(\gl,A_1)\tilde{B}} \norm{\tilde{C}_{\ga_2}}\norm{R(\gl,A_2)(-A_2)^{-\ga_2}x } \\
&\leq M_0 \norm{\tilde{C}_{\ga_2}} \norm{x}\norm{R(\gl,A_1)\tilde{B}} 
}
where $\tilde{C}_{\ga_2}$ is the bounded extension of $\tilde{C}(-A_2)^{\ga_2}$ to $X_2$. Since $\norm{R(i\gw,A_1)}=\Omi(\abs{\gw}^\ga)$, this again implies $\norm{R(i\gw,A_1)\tilde{B}} \norm{\tilde{C}R(i\gw,A_2)} = \Omi(\abs{\gw}^\ga)$.  
If in addition $\dim Y_1<\infty$, we have
\eq{
\MoveEqLeft\sup_{\xi>0} \; \xi \int_{-\infty}^\infty \norm{R(\xi+i\eta,A_1)\tilde{B}}^2 \norm{\tilde{C}R(\xi+i\eta,A_2)x}^2 d\eta \\
&\leq M_0^2\norm{\tilde{C}_{\ga_2}}^2  \sup_{\xi>0} \; \xi \int_{-\infty}^\infty  \norm{R(\xi+i\eta,A_2) \tilde{B}}^2 d\eta <\infty
}
again by Lemma~\ref{lem:unifbddconds}. This concludes that~\eqref{eq:RBCRint} holds if $\gb=0$.

It remains to consider the case where $0<\gb<\ga_1$ and $0<\gg<\ga_2$ satisfy $\gb/\ga_1+\gg/\ga_2\geq  1$. We can choose $0<\gb_0\leq \gb$ and $0<\gg_0\leq \gg$ such that $\gb_0/\ga_1+\gg_0/\ga_2= 1$.
By the Moment Inequality~\cite[Prop. 6.6.2]{haasefuncalc} there exist $M_{{\gb_0}/\ga_1},M_{{\gg_0}/\ga_2}\geq 1$ such that
\eq{
\norm{(-A_1)^{-{\gb_0}} R } \leq M_{{\gb_0}/\ga_1} \norm{R}^{1-{\gb_0}/\ga_1}\norm{(-A_1)^{-\ga_1} R }^{{\gb_0}/\ga_1}\\
\norm{(-A_2)^{-{\gg_0}} Q } \leq M_{{\gg_0}/\ga_2} \norm{Q}^{1-{\gg_0}/\ga_2}\norm{(-A_2)^{-\ga_2} Q }^{{\gg_0}/\ga_2}
}
for any $R\in \Lin(Y,X_1)$ and $Q\in \Lin(Y,X_2)$.
Let $\gl\in \overline{\C^+}$ and for brevity denote $R_1=R(\gl,A_1)$ and $R_2=R(\gl,A_2)$. If  $\tilde{C}_{\gg_0}\in \Lin(X_2,Y)$ is the bounded extension of $\tilde{C}(-A_2)^{\gg_0}$ to $X_2$, and $\tilde{B}_{\gb_0}=(-A_1)^{\gb_0}\tilde{B} \in \Lin(Y,X_1)$, then for every $\gl\in \overline{\C^+}$
  \eq{
  \MoveEqLeft \norm{R(\gl,A_1)\tilde{B}} \norm{\tilde{C}R(\gl,A_2)x} 
  =\norm{  (-A_1)^{-{\gb_0}} R_1 (-A_1)^{\gb_0} \tilde{B}} \norm{\tilde{C} (-A_2)^{\gg_0} (-A_2)^{-{\gg_0}} R_2x} \\
&\leq\norm{  (-A_1)^{-{\gb_0}} R_1 \tilde{B}_{\gb_0}} \norm{\tilde{C}_{\gg_0}} \norm{ (-A_2)^{-{\gg_0}} R_2 x } \\
  &\leq M_{{\gb_0}/\ga_1} \norm{  R_1 \tilde{B}_{\gb_0}}^{1-{\gb_0}/\ga_1} \norm{ (-A_1)^{-\ga_1} R_1 \tilde{B}_{\gb_0}}^{{\gb_0}/\ga_1} \\
  &\qquad \times \norm{  \tilde{C}_{\gg_0}} M_{{\gg_0}/\ga_2} \norm{ R_2 x }^{1-{\gg_0}/\ga_2} \norm{(-A_2)^{-\ga_2} R_2 x }^{{\gg_0}/\ga_2}\\
&\leq M_{{\gb_0}/\ga_1} M_{{\gg_0}/\ga_2} M_0^2 \norm{  \tilde{B}_{\gb_0}}^{{\gb_0}/\ga_1} \norm{  \tilde{C}_{\gg_0}} \norm{x}^{\gg_0/\ga_2} \norm{  R(\gl,A_1) \tilde{B}_{\gb_0}}^{1-{\gb_0}/\ga_1}  \norm{ R(\gl,A_2)x}^{1-{\gg_0}/\ga_2} 
  } 
  which concludes $\norm{R(i\gw,A_1)\tilde{B}} \norm{\tilde{C}R(i\gw,A_2)} = \Omi(\abs{\gw}^\ga)$ since  $\norm{R(i\gw,A_1)}=\Omi(\abs{\gw}^\ga)$,  $\norm{R(i\gw,A_2)}=\Omi(\abs{\gw}^\ga)$, and $1-\gb_0/\ga_1+1-\gg_0/\ga_2 = 1$. 
 If in addition $\dim Y_1<\infty$,
 using the Hölder inequality for $p=1/(1-\gb_0/\ga_1)$ and $q=1/(1-\gg/\ga_2)$ (which satisfy $1/p+1/q=1-\gb_0/\ga_1+1-\gg_0/\ga_2=1$) and denoting $\tilde{M} = M_{{\gb_0}/\ga_1} M_{{\gg_0}/\ga_2} M_0^2 \norm{  \tilde{B}_{\gb_0}}^{{\gb_0}/\ga_1} \norm{  \tilde{C}_{\gg_0}} \norm{x}^{\gg_0/\ga_2}$, we have
\eq{
\MoveEqLeft[1] \sup_{\xi>0}\; \xi \int_{-\infty}^\infty \norm{R(\xi+i\eta,A_1)\tilde{B}}^2 \norm{\tilde{C}R(\xi+i\eta,A_2)x }^2 d\eta \\
&\leq \tilde{M}^2  \sup_{\xi>0}\; \xi \int_{-\infty}^\infty  \norm{R(\xi+i\eta,A_1)\tilde{B}_{\gb_0}}^{2(1-\gb_0/\ga_1)} \norm{R(\xi+i\eta,A_2)x}^{2(1-\gg_0/\ga_2)}  d\eta \\
&\leq \tilde{M}^2  \sup_{\xi>0}\; \left( \xi \int_{-\infty}^\infty  \norm{R(\xi+i\eta,A_1)\tilde{B}_{\gb_0}}^2 d\eta \right)^{\frac{1}{p}} \left( \xi \int_{-\infty}^\infty \norm{R(\xi+i\eta,A_2)x}^2  d\eta \right)^{\frac{1}{q}}\\
&\leq \tilde{M}^2  \left(  \sup_{\xi>0}\;\xi \int_{-\infty}^\infty  \norm{R(\xi+i\eta,A_1)\tilde{B}_{\gb_0}}^2 d\eta \right)^{\frac{1}{p}} \left( \sup_{\xi>0}\; \xi \int_{-\infty}^\infty \norm{R(\xi+i\eta,A_2)x}^2  d\eta \right)^{\frac{1}{q}}<\infty
}
by Lemma~\ref{lem:unifbddconds}. This concludes~\eqref{eq:RBCRint} is true if $\gb,\gg>0$, and thus concludes the proof.
\end{proof}

\begin{proof}[Proof of Theorem~\textup{\ref{thm:polpoltri}}]
  We have from Lemma~\ref{lem:trispec} that $\gs(A)\subset \C^-$. In order to prove that $T(t)$ is polynomially stable with $\ga = \max \set{\ga_1,\ga_2}$, we need to show that $T(t)$ is uniformly bounded and $\norm{R(i\gw,A)} = \Omi(\abs{\gw}^\ga)$.
Assume first that $\gb/\ga_1+\gg/\ga_2\geq 1$.
By Lemma~\ref{lem:trispec} the resolvent operator is of the form 
\eq{
  R(\gl,A) = \pmat{R(\gl,A_1) & R(\gl,A_1)BC R(\gl,A_2)\\0&R(\gl,A_2)} 
}
for every $\gl\in \overline{\C^+}$. We have $\norm{R(i\gw,A_1)} = \Omi(\abs{\gw}^\ga)$, and $\norm{R(i\gw,A_2)} = \Omi(\abs{\gw}^\ga)$ by assumption, and $\norm{R(i\gw,A_1)BC R(i\gw,A_2)} \leq \norm{R(i\gw,A_1)B} \norm{C R(i\gw,A_2)} = \Omi(\abs{\gw}^\ga)$ by Lemma~\ref{lem:BCintprops}. Together these properties conclude that $\norm{R(i\gw,A)}=\Omi(\abs{\gw}^\ga)$.

  If $\gb/\ga_1+\gg/\ga_2>1$, the uniform boundedness of $T(t)$ follows directly from Lemma~\ref{lem:triunifbdd}.

  It remains to show that if $\dim Y <\infty$, then the semigroup $T(t)$ is uniformly bounded whenever $\gb/\ga_1+\gg/\ga_2\geq 1$. Since we already showed that $\norm{R(i\gw,A)}=\Omi(\abs{\gw}^\ga)$, the polynomial stability will then follow from Lemma~\ref{lem:BorTomresolvent}.
For any $x = (x_1,x_2)\in X$ and $\gl\in\C^+$ we have
\eq{
\norm{R(\gl,A)x}^2
&= \norm{R(\gl,A_1)x_1 + R(\gl,A_1)BCR(\gl,A_2)x_2}^2 + \norm{R(\gl,A_2)x_2}^2\\
&\leq 2\norm{R(\gl,A_1)x_1}^2 + 2\norm{ R(\gl,A_1)B}^2\norm{CR(\gl,A_2)x_2}^2 + \norm{R(\gl,A_2)x_2}^2\\
\norm{R(\gl,A)^\ast x}^2
&= \norm{R(\gl,A_1)^\ast x_1}^2 + \norm{\left( R(\gl,A_1)BCR(\gl,A_2) \right)^\ast x_1 + R(\gl,A_2)^\ast x_2}^2\\
&\leq \norm{R(\gl,A_1)^\ast x_1}^2 + 2\norm{ R(\gl,A_2)^\ast C^\ast}^2\norm{B^\ast R(\gl,A_1)^\ast x_1}^2 + 2\norm{R(\gl,A_2)^\ast x_2}^2\\
&\leq \norm{R(\gl,A_1)^\ast x_1}^2 + 2\norm{ R(\overline{\gl},A_2^\ast) C^\ast}^2\norm{B^\ast R(\overline{\gl},A_1^\ast) x_1}^2 + 2\norm{R(\gl,A_2)^\ast x_2}^2.
}
Now
\eq{
\MoveEqLeft \sup_{\xi>0}\; \xi \int_{-\infty}^\infty \left( \norm{R(\xi+i\eta,A)x}^2 + \norm{R(\xi+i\eta,A)^\ast x}^2 \right) d\eta\\
&\leq 2 \sup_{\xi>0}\;\xi \int_{-\infty}^\infty \left( \norm{R(\xi+i\eta,A_1) x_1}^2 + \norm{R(\xi+i\eta,A_1)^\ast x_1}^2 \right) d\eta\\
&\qquad +2 \sup_{\xi>0}\;\xi \int_{-\infty}^\infty \left( \norm{R(\xi+i\eta,A_2) x_2}^2 + \norm{R(\xi+i\eta,A_2)^\ast x_2}^2 \right) d\eta\\
&\qquad + 2 \sup_{\xi>0}\; \xi \int_{-\infty}^\infty \norm{R(\xi+i\eta,A_1)B}^2 \norm{CR(\xi+i\eta,A_2)x_2}^2 d\eta \\
&\qquad + 2 \sup_{\xi>0}\; \xi \int_{-\infty}^\infty \norm{R(\xi-i\eta,A_2^\ast)C^\ast}^2 \norm{B^\ast R(\xi-i\eta,A_1^\ast) x_1}^2 d\eta<\infty,
}
where the first two supremums on the right hand side are finite by Lemma~\ref{lem:unifbddconds} since $T_1(t)$ and $T_2(t)$ are uniformly bounded. The third and the fourth supremums are finite by Lemma~\ref{lem:BCintprops} since $\dim Y<\infty$. Now Lemma~\ref{lem:unifbddconds} concludes that the semigroup $T(t)$ is uniformly bounded, and it is therefore polynomially stable with $\ga= \max \set{\ga_1,\ga_2}$.
\end{proof}

\begin{proof}[Proof of Theorem~\textup{\ref{thm:polpoltribanach}}]
  We have from Lemmas~\ref{lem:trispec} and~\ref{lem:triunifbdd} that $\gs(A)\subset \C^-$ and that the semigroup $T(t)$ is uniformly bounded. We therefore have from~\cite[Cor. 4.2]{ChiTom07} that $T(t)$ is strongly stable.

  Since the semigroups $T_1(t)$ and $T_2(t)$ are polynomially stable, we have from~\cite[Prop 1.3 \& Ex. 1.4]{BatDuyPolStab} that $\norm{R(i\gw,A_1)}=\Omi(\abs{\gw}^{\ga_1})$ and $\norm{R(i\gw,A_2)}=\Omi(\abs{\gw}^{\ga_2})$.
  If $\ga = \ga_1+\ga_2$, then
  \eq{
  \norm{R(i\gw,A_1)B}\norm{CR(i\gw,A_2)} 
  \leq \norm{R(i\gw,A_1)} \norm{B}\norm{C} \norm{R(i\gw,A_2)} 
  = \Omi(\abs{\gw}^\ga),
  }
  which together with Lemma~\ref{lem:trispec} further implies that $\norm{R(i\gw,A)}=\Omi(\abs{\gw}^\ga)$. We now have from~\cite[Thm. 1.5 \& Ex. 1.7]{BatDuyPolStab} that there exists $M\geq 1$ such that
  \eq{
  \norm{T(t)A\inv }\leq M \left( \frac{\ln t}{t} \right)^{1/\ga} 
  }
  for all $t>0$.
\end{proof}

\begin{proof}[Proof of Theorem~\textup{\ref{thm:polexptri}}] 
  Since by Definition~\ref{def:polstab} a polynomially stable semigroup is also strongly stable, we have from~\cite[Lem. 20]{hamalainenpohjolainen10} that the semigroup $T(t)$ is strongly stable. In particular this implies that $T(t)$ is uniformly bounded. By Lemma~\ref{lem:BorTomresolvent} it remains to show that $i\R\subset \rho(A)$ and $\norm{R(i\gw,A)}= \Omi(\abs{\gw}^\ga)$.

  Let $x=(x_1,x_2)^T\in X$ be such that $\norm{x}^2=\norm{x_1}^2 + \norm{x_2}^2=1$. For brevity denote $R_1= R(i\gw,A_1)$ and $R_2= R(i\gw,A_2)$. Now
  \eq{
  \norm{R(i\gw,A)x}^2 &= \Norm{\pmat{R_1 & R_1BC R_2\\0&R_2} \pmat{x_1\\x_2}}^2
  = \norm{R_1 x_1 + R_1BCR_2 x_2 }^2 + \norm{R_2 x_2}^2\\
  &\leq 2 \left( \norm{R_1}^2 \norm{x_1}^2 + \norm{R_1}^2 \norm{BC}^2 \norm{R_2}^2 \norm{x_2}^2 \right) + \norm{R_2}^2 \norm{x_2}^2\\
  &\leq 2 \left( \norm{x_1}^2 + \norm{x_2}^2 \right) \left( \norm{R_1}^2  + \norm{R_1}^2 \norm{BC}^2 \norm{R_2}^2  + \norm{R_2}^2 \right)\\
  &\leq 2 \max \Set{\norm{BC}^2,1} \left( \norm{R_1}^2 (1 +   \norm{R_2}^2)  + \norm{R_2}^2 \right)\\
  & \leq 2 (\norm{BC}^2 + 1) (\norm{R(i\gw,A_1)}^2 + 1) (\norm{R(i\gw,A_2)}^2 + 1) 
  }
  Due to the assumptions and Lemma~\ref{lem:BorTomresolvent} one of the norms $\norm{R(i\gw,A_1)}$ and $\norm{R(i\gw,A_2)}$ is of order $\Omi(\abs{\gw}^\ga)$, and the other is uniformly bounded. This together with the above estimate concludes that $\norm{R(i\gw,A)}=\Omi(\abs{\gw}^\ga)$.
\end{proof}

\section[Full Operator Matrices]{Semigroups Generated By Full Operator Matrices}
\label{sec:fullsys}

In this section we prove the results concerning the semigroup generated by the block operator matrix
\eq{
A= \pmat{A_1&B_1C_2\\B_2C_1&A_2}
}
where $A_1$ and $A_2$ generate polynomially stable semigroups. The operators $B_1\in \Lin(Y_1,X_1)$, $B_2\in \Lin(Y_2,X_2)$, $C_1\in \Lin(X_1,Y_2)$, and $C_2\in \Lin(X_2,Y_1)$ satisfy
\eq{
\ran(B_1)\subset \Dom( (-A_1)^{\gb_1}), \quad \ran(C_1^\ast ) \subset \Dom( (-A_1^\ast)^{\gg_1})\\
\ran(B_2)\subset \Dom( (-A_2)^{\gb_2}), \quad \ran(C_2^\ast ) \subset \Dom( (-A_2^\ast)^{\gg_2})
}
for some $\gb_1,\gb_2,\gg_1,\gg_2\geq 0$.

\begin{lemma} 
  \label{lem:fullspec}
  If $\gl \in \overline{\C^+}$ is such that $1\in\rho(C_2R(\gl,A_2)B_2C_1R(\gl,A_1)B_1)$, then $\gl\in \rho(A)$ and
  \eq{
  R(\gl,A) 
  &= \pmat{R(\gl,A_1)+R(\gl,A_1)B_1C_2S_1(\gl)\inv B_2C_1R(\gl,A_1) &R(\gl,A_1)B_1C_2S_1(\gl)\inv\\ S_1(\gl)\inv B_2C_1R(\gl,A_1)&S_1(\gl)\inv} 
  }
  where 
  \eq{ 
  S_1(\gl)\inv &= R(\gl,A_2) + R(\gl,A_2)B_2C_1R(\gl,A_1)B_1 D_\gl\inv C_2R(\gl,A_2),\\ 
  D_\gl &= I-C_2 R(\gl,A_2)B_2C_1R(\gl,A_1)B_1 .
  }
\end{lemma}

\begin{proof}
  Let $\gl\in \overline{\C^+}$ be such that $1\in\rho(C_2R(\gl,A_2)B_2C_1R(\gl,A_1)B_1)$ and denote $R_1=R(\gl,A_1)$ and $R_2=R(\gl,A_2)$.  
The Schur complement $S_1(\gl)$ of $\gl-A_1$ in 
  \eq{
  \gl-A=\pmat{\gl-A_1&-B_1C_2\\-B_2C_1&\gl-A_2} 
  }
  is
  \eq{
  S_1(\gl) 
  &= \gl-A_2 - B_2C_1R_1B_1C_2 .
  }
  Since $1\in\rho(C_2R_2B_2C_1R_1B_1)$,
  the Shermann--Morrison--Woodbury formula (see, e.g.,~\cite[Lem. 10]{Pau12}) implies that $S_1(\gl)$ is boundedly invertible and
  \eq{
  S_1(\gl)\inv = R_2 + R_2B_2C_1R_1B_1(I-C_2R_2B_2C_1R_1B_1)\inv C_2 R_2
  }
  Since the $\gl-A_1$ and its Schur complement $S_1(\gl)$ are boundedly invertible, we have that $\gl\in \rho(A)$ and the resolvent operator $R(\gl,A)$ is given by
  \eq{
 \MoveEqLeft R(\gl,A)
  = \pmat{I&R_1B_1C_2\\0&I} \pmat{R_1&0\\0&S_1(\gl)\inv}\pmat{I&0\\B_2C_1R_1&I}\\
  &= \pmat{R_1+R_1B_1C_2S_1(\gl)\inv B_2C_1R_1 &R_1B_1C_2S_1(\gl)\inv\\S_1(\gl)\inv B_2C_1R_1&S_1(\gl)\inv} .
  }
\end{proof}

\begin{lemma}
  \label{lem:RBCRintmix}
  Under the assumptions of Theorem~\textup{\ref{thm:polpolfull}}, we have $\norm{R(i\gw,A_k)B_k}\norm{C_l R(i\gw,A_l)} = \Omi(\abs{\gw}^\ga)$ with $\ga = \max \set{\ga_1,\ga_2}$, and
\begin{subequations}
  \label{eq:RBCRintmix}
  \eqn{
  \label{eq:RBCRintmix1}
  &\sup_{\xi>0}\; \xi \int_{-\infty}^\infty \norm{R(\xi+i\eta,A_k)B_k}^2 \norm{C_l R(\xi+i\eta,A_l)x_l}^2 d\eta <\infty\\
  \label{eq:RBCRintmix2}
  &\sup_{\xi>0}\; \xi \int_{-\infty}^\infty \norm{R(\xi+i\eta,A_k)^\ast C_k^\ast}^2 \norm{B_l^\ast R(\xi+i\eta,A_l)^\ast x_l}^2 d\eta <\infty
  } 
\end{subequations}
for every $k,l \in\set{1,2}$.
\end{lemma}

\begin{proof}
  The property $\norm{R(i\gw,A_k)B_k}\norm{C_l R(i\gw,A_l)} = \Omi(\abs{\gw}^\ga)$ follows from Lemma~\ref{lem:BCintprops} since in each of the situations (i)-(iv) the exponents satisfy $\gb_k/\ga_k + \gg_l/\ga_l\geq 1$.

  We have from Lemma~\ref{lem:BCintprops} that for fixed $k,l\in \set{1,2}$ the condition~\eqref{eq:RBCRintmix1} is satisfied if either
  \begin{itemize}
    \item[(a)] $\gb_k\geq \ga_k$, or 
    \item[(b)] $\dim Y_k<\infty$ and $\gb_k/\ga_k+\gg_l/\ga_l\geq 1$.  
  \end{itemize}
It is therefore sufficient to verify that in each of the situations (i)-(iv), for every $k,l\in \set{1,2}$ either (a) or (b) is satisfied.

In the following we list the possible situations with respect to the assumptions (i)-(iv), and the possible combinations of $(k,l)$

\textbf{Assumption (i):} (a) is satisfied for $k=1,2$ (and consequently, for every $(k,l) \in\set{1,2}\times \set{1,2}$).  

 \textbf{Assumption (ii):}
    For $(k,l)$
    \begin{itemize} 
      \item[$(1,1)$] (b) is satisfied since $\dim Y_1<\infty$ and $\gb_1/\ga_1+\gg_1/\ga_1\geq 1$
      \item[$(1,2)$] (b) is satisfied since $\dim Y_1<\infty$ and $\gb_1/\ga_1+\gg_2/\ga_2\geq \gb_1/\ga_1+1\geq 1$
      \item[$(2,l)$] (a) is safisfied since $\gb_2\geq \ga_2$
    \end{itemize}

 \textbf{Assumption (iii):}
    For $(k,l)$
    \begin{itemize} 
      \item[$(1,l)$] (a) is satisfied since $\gb_1\geq \ga_1$.
      \item[$(2,1)$] (b) is satisfied since $\dim Y_2<\infty$ and $\gb_2/\ga_2+\gg_1/\ga_1\geq \gb_2/\ga_2+1\geq 1$
      \item[$(2,2)$] (b) is satisfied since $\dim Y_2<\infty$ and $\gb_2/\ga_2+\gg_2/\ga_2\geq 1$
    \end{itemize}

 \textbf{Assumption (iv):} (b) is satisfied for for every $(k,l) \in\set{1,2}\times \set{1,2}$ since $\dim Y_k<\infty$ and $\gb_k/\ga_k + \gg_l/\ga_l\geq 1$ by assumption.

To show~\eqref{eq:RBCRintmix2}, we note that
\eq{
\MoveEqLeft\sup_{\xi>0}\; \xi \int_{-\infty}^\infty \norm{R(\xi+i\eta,A_k)^\ast C_k^\ast}^2 \norm{B_l^\ast R(\xi+i\eta,A_l)^\ast x_l}^2 d\eta \\
&= \sup_{\xi>0}\; \xi \int_{-\infty}^\infty \norm{R(\xi-i\eta,A_k^\ast)C_k^\ast}^2 \norm{B_l^\ast R(\xi-i\eta,A_l^\ast)x_l}^2 d\eta \\
&= \sup_{\xi>0}\; \xi \int_{-\infty}^\infty \norm{R(\xi+i\eta,A_k^\ast)C_k^\ast}^2 \norm{B_l^\ast R(\xi+i\eta,A_l^\ast)x_l}^2 d\eta .
}
We can apply Lemma~\ref{lem:BCintprops} to operators $A^\ast$, $C_k^\ast$ and $B_k^\ast$ for $k,l\in\set{1,2}$, and see that
for fixed $k,l\in \set{1,2}$ the condition~\eqref{eq:RBCRintmix2} is satisfied if either
  \begin{itemize}
    \item[(a$'$)] $\gg_k\geq \ga_k$, or 
    \item[(b$'$)] $\dim Y_k<\infty$ and $\gb_l/\ga_l+\gg_k/\ga_k\geq 1$.  
  \end{itemize}
  Similarly as above, it can be verified that in every situation (i)-(iv) either (a$'$) or (b$'$) is satisfied.
\end{proof}

We can now conclude this section by presenting the proofs of Theorems~\ref{thm:polpolfull} and~\ref{thm:polexpfull}.

\begin{proof}[Proof of Theorem~\textup{\ref{thm:polpolfull}}]
  By Lemma~\ref{lem:CRBbdd} we can choose $M_1,M_2\geq 1$ such that
  \begin{subequations}
    \label{eq:polpolfCRBbnds}
    \eqn{
    \norm{C_1R(\gl,A_1)B_1} &\leq M_1 \norm{(-A_1)^{\gb_1}B_1} \norm{(-A_1^\ast)^{\gg_1}C_1^\ast}\\
    \norm{C_2R(\gl,A_2)B_2} &\leq M_2 \norm{(-A_2)^{\gb_2}B_2} \norm{(-A_2^\ast)^{\gg_2}C_2^\ast} 
    }
  \end{subequations}
  for all $\gl\in \overline{\C^+}$. If we choose $0<\gd<1/(M_1M_2)$ and if

  \eq{
  \norm{(-A_1)^{\gb_1}B_1} \norm{(-A_1^\ast)^{\gg_1}C_1^\ast} \norm{(-A_2)^{\gb_2}B_2} \norm{(-A_2^\ast)^{\gg_2}C_2^\ast}<\gd,
  }
  then
\eq{
\MoveEqLeft\norm{ C_1R(\gl,A_1)B_1 C_2R(\gl,A_2)B_2}
\leq\norm{ C_1R(\gl,A_1)B_1} \norm{ C_2R(\gl,A_2)B_2}
\leq \gd M_1M_2<1.
}
In particular, this implies that $1\in \rho(C_1R(\gl,A_1)B_1 C_2R(\gl,A_2)B_2)$ and
\eq{
\norm{ D_\gl\inv} 
&=\norm{ \left( I- C_1R(\gl,A_1)B_1 C_2R(\gl,A_2)B_2 \right)\inv} 
=\bigl\|  \sum_{n=0}^\infty \left(C_1R(\gl,A_1)B_1 C_2R(\gl,A_2)B_2 \right)^n\bigr\|\\
 &\leq\sum_{n=0}^\infty \norm{ C_1R(\gl,A_1)B_1 C_2R(\gl,A_2)B_2 }^n 
 \leq \sum_{n=0}^\infty \; ( \gd M_1M_2 )^n 
 = \frac{1}{1-\gd M_1M_2}<\infty
}
for all $\gl\in \overline{\C^+}$. Lemma~\ref{lem:fullspec} now concludes that $\overline{\C^+}\subset \rho(A)$ and gives a formula for the resolvent $R(\gl,A)$ for $\gl\in \overline{\C^+}$.

To prove uniform boundedness of $T(t)$ using Lemma~\ref{lem:unifbddconds}, we need to estimate norms $\norm{R(\gl,A)x}$ and $\norm{R(\gl,A)^\ast x}$ for $x = (x_1,x_2)^T\in X$ 
and $\gl\in \overline{\C^+}$.
Let $\gl\in \overline{\C^+}$ and denote
$R_1=R(\gl,A_1)$, $R_2=R(\gl,A_2)$, and $D_\gl= I- C_1R(\gl,A_1)B_1 C_2R(\gl,A_2)B_2$.
If $M_D = 1/(1-\gd M_1M_2)$, we saw that $\norm{D_\gl\inv}\leq M_D$. 
If we choose $\tilde{M}_1 = M_1 \norm{(-A_1)^{\gb_1}B_1} \norm{(-A_1^\ast)^{\gg_1}C_1^\ast}$ and $\tilde{M}_2 =  M_2 \norm{(-A_2)^{\gb_2}B_2} \norm{(-A_2^\ast)^{\gg_2}C_2^\ast} $, then~\eqref{eq:polpolfCRBbnds} imply
  \eq{
  \norm{C_1R(\gl,A_1)B_1} \leq \tilde{M}_1 
  \qquad \mbox{and} \qquad
\norm{C_2R(\gl,A_2)B_2}&\leq \tilde{M}_2.
  } 
  In the estimates we use the scalar inequalities $(a+b)^2\leq 2(a^2+b^2)$ and $(a+b+c)^2\leq 3(a^2+b^2+c^2) $ for $a,b,c\geq 0$.
We have
\eq{
\MoveEqLeft \norm{R(\gl,A) x}^2
=  \Norm{\pmat{R_1x_1+R_1B_1C_2S_1(\gl)\inv B_2C_1R_1 x_1 +R_1B_1C_2S_1(\gl)\inv x_2\\S_1(\gl)\inv B_2C_1R_1x_1+S_1(\gl)\inv x_2} }^2\\
&=  \norm{R_1x_1+R_1B_1C_2S_1(\gl)\inv B_2C_1R_1 x_1 +R_1B_1C_2S_1(\gl)\inv x_2}^2 \\
&\qquad+ \norm{S_1(\gl)\inv B_2C_1R_1x_1+S_1(\gl)\inv x_2}^2\\
&\leq 3 \norm{R_1x_1}^2+ 3\norm{R_1B_1C_2S_1(\gl)\inv B_2C_1R_1 x_1}^2 + 3\norm{R_1B_1C_2S_1(\gl)\inv x_2}^2 \\
&\qquad+ 2\norm{S_1(\gl)\inv B_2C_1R_1x_1}^2 + 2\norm{S_1(\gl)\inv x_2}^2,
}
where the terms on the right-hand side can be further estimated by
\eq{
\MoveEqLeft \norm{R_1B_1C_2S_1(\gl)\inv B_2C_1R_1 x_1}
=\norm{R_1B_1C_2(R_2 + R_2B_2C_1R_1B_1D_\gl\inv C_2R_2)B_2C_1R_1 x_1}\\
&\leq \norm{R_1B_1} \left( \norm{C_2 R_2B_2}  + \norm{C_2R_2B_2} \norm{C_1R_1B_1} \norm{D_\gl\inv} \norm{ C_2R_2 B_2}  \right)\norm{C_1R_1 x_1}\\
&\leq \left( \tilde{M}_2  + \tilde{M}_1 \tilde{M}_2^2 M_D  \right) \norm{R_1B_1} \norm{C_1R_1x_1}
} 
\eq{
\MoveEqLeft \norm{R_1B_1C_2S_1(\gl)\inv x_2}
= \norm{R_1B_1C_2 (R_2 + R_2B_2C_1R_1B_1D_\gl\inv C_2R_2) x_2} \\
& \leq \norm{R_1B_1} \norm{C_2 R_2  x_2} + \norm{R_1B_1} \norm{C_2  R_2B_2 } \norm{C_1R_1B_1} \norm{D_\gl\inv} \norm{ C_2R_2 x_2} \\
& \leq \left( 1+ \tilde{M}_1 \tilde{M}_2 M_D \right) \norm{R_1B_1} \norm{C_2 R_2 x_2}
} 
\eq{
\MoveEqLeft \norm{S_1(\gl)\inv B_2C_1R_1x_1}
=\norm{(R_2 + R_2B_2C_1R_1B_1D_\gl\inv C_2R_2) B_2C_1R_1x_1}\\
&\leq \norm{R_2  B_2} \norm{C_1R_1x_1} +\norm{ R_2B_2} \norm{C_1R_1B_1} \norm{D_\gl\inv} \norm{ C_2R_2 B_2 } \norm{C_1R_1x_1}\\
&\leq (1+ \tilde{M}_1 \tilde{M}_2 M_D)  \norm{R_2  B_2} \norm{C_1R_1 x_1} 
} 
\eq{
\MoveEqLeft \norm{S_1(\gl)\inv x_2}
= \norm{R_2 x_2 + R_2B_2C_1R_1B_1D_\gl\inv C_2R_2 x_2}\\
&\leq \norm{R_2 x_2 } + \norm{R_2B_2} \norm{C_1R_1B_1} \norm{D_\gl\inv} \norm{ C_2R_2 x_2}\\
&\leq \norm{R_2 x_2 } + \tilde{M}_1 M_D  \norm{R_2B_2}  \norm{ C_2R_2 x_2}
}
Denote $M_{\text{tot}} = \max \set{\tilde{M}_1 M_D, 1+\tilde{M}_1 \tilde{M}_2 M_D, \tilde{M}_2+\tilde{M}_1 \tilde{M}_2^2 M_D}$. Combining the above estimates 
yields 
\eq{
\norm{R(\gl,A)x}^2
&\leq 3 \norm{R_1x_1}^2+ 3\norm{R_1B_1C_2S_1(\gl)\inv B_2C_1R_1 x_1}^2 + 3\norm{R_1B_1C_2S_1(\gl)\inv x_2}^2 \\
&\qquad+ 2\norm{S_1(\gl)\inv B_2C_1R_1x_1}^2 + 2\norm{S_1(\gl)\inv x_2}^2\\
& \leq 3 \norm{R_1x_1}^2+ 3 M_{\text{tot}}^2  \norm{R_1B_1}^2 \norm{C_1R_1 x_1}^2 
+ 3M_{\text{tot}}^2  \norm{R_1B_1}^2 \norm{C_2R_2 x_2}^2\\
&\qquad + 2M_{\text{tot}}^2  \norm{R_2B_2}^2 \norm{C_1R_1 x_1}^2
+ 4\norm{R_2 x_2 }^2 + 4 M_{\text{tot}}^2 \norm{R_2B_2}^2  \norm{ C_2R_2 x_2}^2\\
& \leq
3 \norm{R_1x_1}^2+ 4\norm{R_2 x_2 }^2 + 
4 M_{\text{tot}}^2  \sum_{k,l=1,2} \norm{R_kB_k}^2 \norm{C_lR_lx_l}^2 .
}
Lemmas~\ref{lem:unifbddconds} and~\ref{lem:RBCRintmix} thus conclude
\eq{
\MoveEqLeft \sup_{\xi>0}\; \xi \int_{-\infty}^\infty \norm{R(\xi+i\eta,A)x}^2 d\eta 
\leq 3\sup_{\xi>0}\; \xi \int_{-\infty}^\infty \norm{R(\xi+i\eta,A_1)x_1}^2 d\eta \\
& \qquad+4 \sup_{\xi>0}\; \xi \int_{-\infty}^\infty \norm{R(\xi+i\eta,A_2)x_2}^2 d\eta \\
&\qquad +4 M_{\text{tot}}^2 \sum_{k,l=1,2} \sup_{\xi>0}\; \xi \int_{-\infty}^\infty \norm{R(\xi+i\eta,A_k)B_k}^2 \norm{C_lR(\xi+i\eta,A_l)x_l}^2 d\eta
<\infty.
}
Furthermore, the same estimates also show that
\eq{
\norm{R(i\gw,A)}^2
&\leq 3 \norm{R(i\gw,A_1)}^2 + 4\norm{R(i\gw,A_2)}^2  \\
& \quad + 4 M_{\text{tot}}^2  \sum_{k,l=1,2} \norm{R(i\gw,A_k)B_k}^2 \norm{C_lR(i\gw,A_l)}^2 .
}
This immediately implies $\norm{R(i\gw,A)}= \Omi(\abs{\gw}^\ga)$ with $\ga = \max \set{\ga_1,\ga_2}$, since we have $\norm{R(i\gw,A_k)B_k} \norm{C_lR(i\gw,A_l)}  =\Omi(\abs{\gw}^\ga)$ by Lemma~\ref{lem:RBCRintmix}.

Similarly, we can estimate the norm of $R(\gl,A)^\ast x$ by
\eq{
\MoveEqLeft\norm{R(\gl,A)^\ast x}^2= 
\Norm{\pmat{R_1^\ast x_1+\left( R_1B_1C_2S_1(\gl)\inv B_2C_1R_1 \right)^\ast x_1 + \left( S_1(\gl)\inv B_2C_1R_1 \right)^\ast x_2 \\\left( R_1B_1C_2S_1(\gl)\inv \right)^\ast x_1 +\left( S_1(\gl)\inv \right)^\ast x_2}}^2\\
&=\norm{R_1^\ast x_1+( R_1B_1C_2S_1(\gl)\inv B_2C_1R_1 )^\ast x_1 + ( S_1(\gl)\inv B_2C_1R_1 )^\ast x_2 }^2 \\
&\qquad + \norm{( R_1B_1C_2S_1(\gl)\inv )^\ast x_1 +( S_1(\gl)\inv )^\ast x_2}^2\\
&\leq 3\norm{R_1^\ast x_1}^2+ 3\norm{( R_1B_1C_2S_1(\gl)\inv B_2C_1R_1 )^\ast x_1}^2  + 3\norm{( S_1(\gl)\inv B_2C_1R_1 )^\ast x_2 }^2 \\
&\qquad + 2\norm{( R_1B_1C_2S_1(\gl)\inv )^\ast x_1}^2 + 2\norm{( S_1(\gl)\inv )^\ast x_2}^2.
} 
Since
\eq{
(S(\gl)\inv)^\ast  
= R_2^\ast + R_2^\ast C_2^\ast (C_1R_1B_1)^\ast (D_\gl\inv)^\ast B_2^\ast R_2^\ast,
}
and $ \norm{B_2^\ast R_2^\ast C_2^\ast} =  \norm{(C_2R_2B_2)^\ast} =\norm{C_2R_2B_2} $,
we have 
\eq{
\MoveEqLeft\norm{( R_1B_1C_2S_1(\gl)\inv B_2C_1R_1 )^\ast x_1}
=\norm{ R_1^\ast C_1^\ast B_2^\ast (S_1(\gl)\inv)^\ast C_2^\ast B_1^\ast R_1^\ast x_1}\\
&=\norm{ R_1^\ast C_1^\ast B_2^\ast (R_2^\ast + R_2^\ast C_2^\ast (C_1R_1B_1)^\ast (D_\gl\inv)^\ast B_2^\ast R_2^\ast) C_2^\ast B_1^\ast R_1^\ast x_1}\\
&\leq  \norm{ R_1^\ast C_1^\ast} \left( \norm{C_2R_2B_2} + \norm{C_2R_2B_2} \norm{C_1R_1B_1} \norm{D_\gl\inv}\norm{C_2R_2B_2} \right) \norm{B_1^\ast R_1^\ast x_1}\\
&\leq (\tilde{M}_2 + \tilde{M}_1 \tilde{M}_2 M_D) \norm{ R_1^\ast C_1^\ast}  \norm{B_1^\ast R_1^\ast x_1}
} 
\eq{
\MoveEqLeft\norm{( S_1(\gl)\inv B_2C_1R_1 )^\ast x_2 }
=\norm{ R_1^\ast C_1^\ast B_2^\ast (R_2^\ast + R_2^\ast C_2^\ast (C_1R_1B_1)^\ast (D_\gl\inv)^\ast B_2^\ast R_2^\ast)  x_2}\\
&\leq \norm{ R_1^\ast C_1^\ast} \norm{ B_2^\ast R_2^\ast x_2} + \norm{ R_1^\ast C_1^\ast} \norm{ C_2R_2B_2 } \norm{C_1R_1B_1} \norm{D_\gl\inv} \norm{B_2^\ast R_2^\ast  x_2}\\
&\leq (1+ \tilde{M}_1 \tilde{M}_2 M_D) \norm{ R_1^\ast C_1^\ast} \norm{ B_2^\ast R_2^\ast x_2} 
} 
\eq{
\MoveEqLeft \norm{( R_1B_1C_2S_1(\gl)\inv )^\ast x_1 }
=\norm{ (R_2^\ast + R_2^\ast C_2^\ast (C_1R_1B_1)^\ast (D_\gl\inv)^\ast B_2^\ast R_2^\ast) C_2^\ast B_1^\ast R_1^\ast x_1}\\
&\leq\norm{ R_2^\ast C_2^\ast} \norm{ B_1^\ast R_1^\ast x_1} + \norm{R_2^\ast C_2^\ast} \norm{C_1R_1B_1} \norm{D_\gl\inv} \norm{C_2R_2B_2} \norm{B_1^\ast R_1^\ast x_1}\\
&\leq (1+ \tilde{M}_1 \tilde{M}_2 M_D)\norm{ R_2^\ast C_2^\ast} \norm{ B_1^\ast R_1^\ast x_1} 
}
\eq{
\MoveEqLeft\norm{( S_1(\gl)\inv )^\ast x_2 }
=\norm{ R_2^\ast x_2 + R_2^\ast C_2^\ast (C_1R_1B_1)^\ast (D_\gl\inv)^\ast B_2^\ast R_2^\ast  x_2}\\
&\leq\norm{ R_2^\ast x_2} + \norm{R_2^\ast C_2^\ast} \norm{C_1R_1B_1} \norm{D_\gl\inv} \norm{B_2^\ast R_2^\ast  x_2}\\
&\leq\norm{ R_2^\ast x_2} + \tilde{M}_1M_D \norm{R_2^\ast C_2^\ast} \norm{B_2^\ast R_2^\ast  x_2}
}
Similarly as in the case of $\norm{R(\gl,A)x}$, we therefore have
\eq{
\norm{R(\gl,A)^\ast x}^2
&\leq 3\norm{R_1^\ast x_1}^2+ 3\norm{( R_1B_1C_2S_1(\gl)\inv B_2C_1R_1 )^\ast x_1}^2  + 3\norm{( S_1(\gl)\inv B_2C_1R_1 )^\ast x_2 }^2 \\
&\qquad + 2\norm{( R_1B_1C_2S_1(\gl)\inv )^\ast x_1}^2 + 2\norm{( S_1(\gl)\inv )^\ast x_2}^2\\
& \leq 3 \norm{R_1^\ast x_1}^2+ 3 M_{\text{tot}}^2  \norm{R_1^\ast C_1^\ast }^2\norm{B_1^\ast R_1^\ast x_1}^2  
+ 3M_{\text{tot}}^2   \norm{R_1^\ast C_1^\ast }^2 \norm{B_2^\ast R_2^\ast x_2}^2\\
&\qquad + 2M_{\text{tot}}^2  \norm{R_2^\ast C_2^\ast}^2\norm{B_1^\ast R_1^\ast x_1}^2 
+ 4\norm{R_2^\ast x_2 }^2 + 4 M_{\text{tot}}^2 \norm{ R_2^\ast C_2^\ast}^2  \norm{ B_2^\ast R_2^\ast x_2}^2\\
&\leq 3 \norm{R_1^\ast x_1}^2+4 \norm{R_2^\ast x_2}^2+ 4 M_{\text{tot}}^2 \sum_{k,l=1,2}  \norm{ R_k^\ast C_k^\ast}^2  \norm{ B_l^\ast R_l^\ast x_l}^2.
}
Lemmas~\ref{lem:unifbddconds} and~\ref{lem:RBCRintmix} now imply
\eq{
\MoveEqLeft \sup_{\xi>0}\; \xi \int_{-\infty}^\infty \norm{R(\xi+i\eta,A)^\ast x}^2 d\eta 
\leq 3\sup_{\xi>0}\; \xi \int_{-\infty}^\infty \norm{R(\xi+i\eta,A_1)^\ast x_1}^2 d\eta \\
& \qquad+4 \sup_{\xi>0}\; \xi \int_{-\infty}^\infty \norm{R(\xi+i\eta,A_2)^\ast x_2}^2 d\eta \\
&\qquad +4 M_{\text{tot}}^2 \sum_{k,l=1,2} \sup_{\xi>0}\; \xi \int_{-\infty}^\infty \norm{R(\xi+i\eta,A_k)^\ast C_k^\ast}^2 \norm{B_l^\ast R(\xi+i\eta,A_l)^\ast x_l}^2 d\eta
<\infty.
}
By Lemma~\ref{lem:unifbddconds} we finally have that the semigroup $T(t)$ is uniformly bounded. This concludes that $T(t)$ is polynomially stable with $\ga=\max \set{\ga_1,\ga_2}$.
\end{proof}

\begin{proof}[Proof of Theorem~\textup{\ref{thm:polexpfull}}]
  Since $T_1(t)$ is exponentially stable, we have $\sup_{\gl\in \overline{\C^+}}\norm{R(\gl,A_1)}<\infty$.
  Because of this and by Lemma~\ref{lem:CRBbdd} we can choose $M_1,M_2\geq 1$ such that
  \eq{
  \norm{C_1R(\gl,A_1)B_1} &\leq M_1 \norm{B_1} \norm{C_1}\\
  \norm{C_2R(\gl,A_2)B_2} &\leq M_2 \norm{(-A_2)^{\gb_2}B_2} \norm{(-A_2^\ast)^{\gg_2}C_2^\ast} 
  }
  for all $\gl\in \overline{\C^+}$. If we choose $0<\gd<1/(M_1M_2)$ and if 
  \eq{
  \norm{B_1} \norm{C_1^\ast} \norm{(-A_2)^{\gb_2}B_2} \norm{(-A_2^\ast)^{\gg_2}C_2^\ast}<\gd,
  }
  then
\eq{
\MoveEqLeft\norm{ C_1R(\gl,A_1)B_1 C_2R(\gl,A_2)B_2}
\leq\norm{ C_1R(\gl,A_1)B_1} \norm{ C_2R(\gl,A_2)B_2}
\leq \gd M_1M_2<1.
}
As in the proof of Theorem~\ref{thm:polpolfull} we can now see that $\gs(A)\subset \C^-$, and $\norm{D_\gl\inv}$ is uniformly bounded for $\gl\in \overline{\C^+}$.

If we can verify that under the assumptions of the theorem the conditions~\eqref{eq:RBCRintmix} are satisfied for every $k,l\in \set{1,2}$, then the uniform boundedness of $T(t)$ follows directly from the estimates made in the proof of Theorem~\ref{thm:polpolfull}.

If $k = 1$ and $l=1,2$, we have that 
\eq{
\MoveEqLeft \sup_{\xi>0}\; \xi \int_{-\infty}^\infty \norm{R(\xi+i\eta,A_1)B_1}^2 \norm{C_l R(\xi+i\eta,A_l)x_l}^2 d\eta\\
&\leq \norm{B_1}^2 \norm{C_l}^2 \Bigl(\sup_{\gl\in  \overline{\C^+}} \norm{R(\gl,A)} \Bigr)^2 \sup_{\xi>0}\; \xi \int_{-\infty}^\infty  \norm{ R(\xi+i\eta,A_l)x_l}^2 d\eta <\infty
}
by Lemma~\ref{lem:unifbddconds}.
On the other hand, if $k=l=2$, then~\eqref{eq:RBCRintmix1} follows directly from Lemma~\ref{lem:BCintprops} and our assumptions. Finally, we need to consider the case where $k=2$ and $l=1$. If $\gb_2\geq \ga_2$, then $\norm{R(\gl,A_2)B_2}\leq \norm{R(\gl,A_2)(-A_2)^{-\ga_2}}\norm{(-A_2)^{\ga_2-\gb_2}} \norm{(-A_2)^{\gb_2}B_2}$ for all $\gl\in \overline{\C^+}$ and
\eq{
\MoveEqLeft[1] \sup_{\xi>0}\; \xi \int_{-\infty}^\infty \norm{R(\xi+i\eta,A_2)B_2}^2 \norm{C_1 R(\xi+i\eta,A_1)x_1}^2 d\eta\\
&\leq \norm{(-A_2)^{\gb_2}B_2}^2 \norm{(-A_2)^{\ga_2-\gb_2}}^2 \norm{C_1}^2 \Bigl(\sup_{\gl\in  \overline{\C^+}} \norm{R(\gl,A_2)(-A_2)^{-\ga_2}} \Bigr)^2 \\
&\quad\times \sup_{\xi>0}\; \xi \int_{-\infty}^\infty  \norm{ R(\xi+i\eta,A_1)x_1}^2 d\eta <\infty 
}
by Lemma~\ref{lem:unifbddconds}.
On the other hand, if $\dim Y_2<\infty$, then
\eq{
\MoveEqLeft[1] \sup_{\xi>0}\; \xi \int_{-\infty}^\infty \norm{R(\xi+i\eta,A_2)B_2}^2 \norm{C_1 R(\xi+i\eta,A_1)x_1}^2 d\eta\\
&\leq  \norm{C_1}^2 \norm{x_1}^2 \Bigl(\sup_{\gl\in  \overline{\C^+}} \norm{R(\gl,A_1)} \Bigr)^2 \sup_{\xi>0}\; \xi \int_{-\infty}^\infty  \norm{ R(\xi+i\eta,A_2)B_2}^2 d\eta <\infty 
}
again by Lemma~\ref{lem:unifbddconds}.

It remains to show that $\norm{R(i\gw,A)}=\Omi(\abs{\gw}^{\ga_2})$.
The estimates made in the proof of Theorem~\ref{thm:polpolfull} show that
\eq{
\norm{R(i\gw,A)}^2
&\leq 3 \norm{R(i\gw,A_1)}^2 + 4\norm{R(i\gw,A_2)}^2  \\
& \quad + 4 M_{\text{tot}}^2  \sum_{k=1,2} \norm{R(i\gw,A_k)B_k}^2 \norm{C_lR(i\gw,A_l)}^2 .
}
This implies $\norm{R(i\gw,A)}= \Omi(\abs{\gw}^{\ga_2})$, since  in the case where $k=1$ or $l=1$, one of the resolvents is uniformly bounded, and we clearly have $\norm{R(i\gw,A_k)B_k} \norm{C_lR(i\gw,A_l)}  =\Omi(\abs{\gw}^{\ga_2})$. In the case $k=l=2$ the same conclusion follows from Lemma~\ref{lem:BCintprops}.  
\end{proof}

\section{Examples on Optimality of the Results}
\label{sec:optex}

In this section we present two examples to illustrate the optimality of the conditions in the results presented in Section~\ref{sec:compstab}.
In the first example we show that the condition $\gb/\ga_1+\gg/\ga_2\geq 1$ is in general an optimal condition for the polynomial stability of a semigroup generated by a triangular block operator matrix.

\begin{example}
  Let $A_1: \Dom(A_1)\subset X_1\rightarrow X_1$ generate a semigroup $T_1(t)$ such that $T_1(t)$ is polynomially stable with $\ga>0$, but 
\eqn{
\label{eq:polpolexass}
\sup_{t>0} t\,\norm{T_1(t)(-A_1)^{-\tilde{\ga}}} = \infty \qquad \mbox{for every} \quad 0\leq \tilde{\ga}<\ga. 
} 
  Choose $X_2=X_1$, $A_2=A_1$, $Y = X_1$, $B = (-A_1)^{-\gb}\in \Lin(X_1)$ and $C= (-A_1)^{-\gg}\in \Lin(X_1)$. Consider the triangular block operator matrix
  \eq{
  A=\pmat{A_1&BC\\0&A_1} = \pmat{A_1& (-A_1)^{-(\gb+\gg)}\\0&A_1}.
  }
  A direct computation shows that the semigroup $T(t)$ generated by $A$ is of the form
  \eq{
  A = \pmat{T_1(t)&tT_1(t)(-A_1)^{-(\gb+\gg)}\\ 0&T_1(t)}.
  }
Since $T_1(t)$ is uniformly bounded, the semigroup $T(t)$ is uniformly bounded if and only if
\eqn{
\label{eq:optexunifbdd}
\sup_{t>0} \; t \,\norm{T_1(t)(-A_1)^{-(\gb+\gg)}}<\infty.
}
Our assumption~\eqref{eq:polpolexass} shows that if $\gb/\ga_1+\gg/\ga_2< 1$, or equivalently $\gb+\gg<\ga$, the semigroup $T(t)$ is not uniformly bounded and it is therefore unstable. This concludes that the condition $\gb/\ga_1+\gg/\ga_2\geq 1$ is in general an optimal condition for the exponents.

On the other hand, if $\gb/\ga_1+\gg/\ga_2\geq 1$, then $T(t)$ is uniformly bounded due the polynomial stability of $A_1$ and~\eqref{eq:optexunifbdd}.
Moreover, the resolvent operator satisfies
\eq{
R(i\gw,A)= \pmat{R(i\gw,A_1)& R(i\gw,A_1)^2(-A_1)^{-(\gb+\gg)}\\0&R(i\gw,A_1)},
}
where 
\eq{
\norm{R(i\gw,A_1)^2(-A_1)^{-(\gb+\gg)}}
\leq\norm{R(i\gw,A_1)} \norm{R(i\gw,A_1)(-A_1)^{-\ga}} \norm{(-A_1)^{\ga-(\gb+\gg)}}
= \Omi(\abs{\gw}^\ga)
}
by Lemma~\ref{lem:BorTomresolvent}. This immediately implies $\norm{R(i\gw,A)}=\Omi(\abs{\gw}^\ga)$, and thus concludes that $T(t)$ is polynomially stable with $\ga$.  
In conclusion, in this example $\gb/\ga_1+\gg/\ga_2\geq 1$ is sufficient
for polynomial stability of $T(t)$ despite the fact that $Y$ is infinite-dimensional.
\end{example}

The second example shows that $\gb_2+\gg_2\geq \ga_2$ in Theorem~\ref{thm:polexpfull} is in general an optimal condition for the exponents.

\begin{example}
  \label{ex:exppol}
  Let $X_1=X_2=\lp[2](\C)$, and consider  
  \eq{
  A_1= \sum_{k=1}^\infty \left( -\gs + ik \right)\iprod{\cdot}{\phi_k}\phi_k,
  \qquad \mbox{and} \qquad
  A_2= \sum_{k=1}^\infty \left( -\frac{1}{k^{\ga_2}} + ik \right)\iprod{\cdot}{\phi_k}\phi_k
  }
  where $\ga_2>0$, $\gs>0$ and $\set{\phi_k}_{k=1}^\infty$ is the Euclidean basis of $X_1$. 
  The operators $A_1$ and $A_2$ generate semigroups $T_1(t)$ and $T_2(t)$, respectively, such that $T_1(t)$ is exponentially stable and $T_2(t)$ is polynomially stable with $\ga_2$. Choose $Y_1=Y_2=\C$ and for some fixed $n\in\N$ let
$B_1 = \gs  \phi_n$, $ B_2 = n^{-\ga_2/2} \phi_n$, $C_1 = \iprod{\cdot}{\phi_n}$, and $C_2 = n^{-\ga_2/2}\iprod{\cdot}{\phi_n}$.
  The block operator matrix $A$ on $X=X_1\times X_2$ is then given by
  \eq{
  A = \pmat{A_1 & \frac{\gs}{n^{\ga_2/2}}\iprod{\cdot}{\phi_n}\phi_n\\\frac{1}{n^{\ga_2/2}} \iprod{\cdot}{\phi_n}\phi_n &A_2} .
  }
  We have $\ran(B_2) = \ran(C_2^\ast) = \Span\set{\phi_n}\subset \Dom( (-A_2)^\infty)$, and therefore the range conditions on $B_2$ and $C_2$ are satisfied for any choices of the exponents $\gb_2,\gg_2\geq 0$. We will show that if $\gb_2+\gg_2<\ga_2$, then
  for any $\gd>0$ we can choose $n\in\N$ in such a way that the semigroup $T(t)$ is unstable even though
  \eq{
  \norm{B_1} \cdot
  \norm{C_1} \cdot
  \norm{(-A_2)^{\gb_2}B_2} \cdot
  \norm{(-A_2^\ast)^{\gg_2}C_2^\ast} <\gd.
  }
  To this end, let $\gd>0$ be arbitrary.
  A direct computation yields $\norm{B_1}=\gs$, $\norm{C_1}=1$, $\norm{(-A_2)^{\gb_2}B_2}=n^{\gb_2-\ga_2/2}$ and $\norm{(-A_2^\ast)^{\gg_2}C_2^\ast}=n^{\gg_2-\ga_2/2}$. Since $\gb_2+\gg_2<\ga_2$ by assumption, the product
  \eq{
  \norm{B_1} \cdot
  \norm{C_1} \cdot
  \norm{(-A_2)^{\gb_2}B_2} \cdot
  \norm{(-A_2^\ast)^{\gg_2}C_2^\ast} 
  = \gs n^{\gb_2+\gg_2-\ga_2}
  }
  can be made smaller than $\gd>0$ by choosing a sufficiently large $n\in \N$. It should be noted that also both of the operator norms $\norm{B_1C_2}=\gs n^{-\ga_2/2}$ and $\norm{B_2C_1}=n^{-\ga_2/2}$ can be made arbitrarily small by choosing a large enough $n\in\N$.

  To show that $T(t)$ is unstable, consider the operator 
  \eq{
  \gl-A = \pmat{\gl-A_1 & -\frac{\gs}{n^{\ga_2/2}} \iprod{\cdot}{\phi_n}\phi_n\\ -\frac{1}{n^{\ga_2/2}} \iprod{\cdot}{\phi_n}\phi_n &\gl-A_2} 
  }
  for  $\gl\in \overline{\C^+}$.
  The Schur complement $S_1(\gl)$ of $\gl-A_1$ in $\gl-A$ is 
  \eq{
  S_1(\gl) 
  = \gl - A_2 - \frac{\gs}{n^{\ga_2}} \iprod{R(\gl,A_1)\phi_n}{\phi_n}\iprod{\cdot}{\phi_n}\phi_n.
  }
  Since $\gs(A_1)\subset \C^-$, we have that $i\gw\in i\R$ is an eigenvalue of $A$ if and only if $0$ is an eigenvalue of $S_1(i\gw)$. But a direct computation shows that
  \eq{
  S_1(in)\phi_n
  = in \phi_n- \bigl( -\gs + in \bigr)\phi_n - \frac{\gs}{n^{\ga_2}} \cdot \frac{1}{in+1/n^{\ga_2} -in}\phi_n
  =  \gs\phi_n   - \frac{\gs}{n^{\ga_2}} n^{\ga_2} \phi_n
  =0.
  }
  This concludes that $\gl=in\in i\R$ is an eigenvalue of $A$, and thus the semigroup generated by $A$ is unstable for all choices of $n\in\N$.  
\end{example}

\section{Connected Wave Equations}
\label{sec:waveex}

In this section we use the results in Section~\ref{sec:compstab} to study the stability properties of a system consisting of two connected wave equations,
  \begin{subequations}
    \label{eq:waveex}
    \eqn{
    \label{eq:waveex1}
    v_{tt}(z,t) - \Delta v(z,t) + a(z) v_t(z,t)&=B_0C_0 w(\cdot,t)\\
    \label{eq:waveex2}
    w_{tt}(r,t)-w_{rr}(r,t)+(1-r)u(t)&=0 
    }
  \end{subequations}
  on $z\in \Omega= (0,\pi)\times (0,\pi)\subset \R^2$ and $r\in (0,1)$, respectively, and
with boundary and initial conditions
\begin{subequations}
  \label{eq:wavebc}
  \eq{
  v(z,t)&=0 \qquad z\in \partial \Omega\\
  v(z,0)&=v_0(z), \quad v_t(z,0)=v_1(z) \\
  w(0,t)&=w(1,t)=0\\
  w(r,0)&=w_0(r), \quad w_t(r,0)=w_1(r).
  }
\end{subequations}
The equation~\eqref{eq:waveex1} is a two-dimensional wave equation with local viscous damping term $a(z)v_t(z,t)$~\cite[Sec 3, Ex. 3]{LiuRao05}. The function $a(z)$ is chosen as
\eq{
a(z) = 
\left\{
\begin{array}{ll}
  1&0\leq z_1\leq 1\\
  0&\mbox{otherwise}
\end{array}
\right.
}
for $z=(z_1,z_2)\in \Omega$ (see Figure~\ref{fig:wavedom}). The function $u(t)$ in~\eqref{eq:waveex2} is chosen in such a way that the one-dimensional subsystem is polynomially stable. This is done in Section~\ref{sec:wave1D}.
Our main aim in this example is to derive conditions for the operators $B_0$ and $C_0$ in the coupling between the equations so that the connected system~\eqref{eq:waveex} is polynomially stable.
To accomplish this, we will write~\eqref{eq:waveex} as a triangular system
\eqn{
\label{eq:waveextrisys}
\ddb{t} \pmat{x_1\\x_2} = \pmat{A_1&BC\\0&A_2} \pmat{x_1\\x_2},
}
on a suitable space $X=X_1\times X_2$, and subsequently use Theorem~\ref{thm:polpoltri} to study the stability of the semigroup generated by its system operator.

\begin{figure}[ht]
  \begin{center}
    \includegraphics[scale=.9]{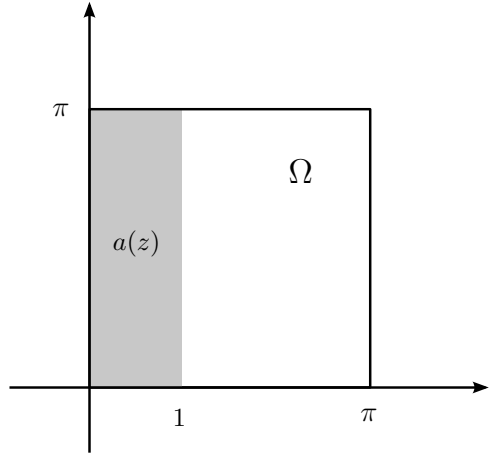}
  \end{center}
  \caption{The domain $\Omega = (0,\pi)\times (0,\pi)\subset \R^2$}
  \label{fig:wavedom}
\end{figure}

\subsection{The Two-Dimensional System}
\label{sec:wave2D}

The equation~\eqref{eq:waveex1} with the boundary conditions in~\eqref{eq:wavebc} can be written as a first order linear system on a Hilbert space $X=H_0^1(\Omega)\times \Lp[2](\Omega)$ with inner product $\iprod{x}{y}_X = \iprod{\nabla x_1}{\nabla y_1}_{\Lp[2](\Omega)^2} + \iprod{ x_2}{ y_2}_{\Lp[2](\Omega)}$ by choosing (see~\cite[Sec. 3, Ex. 3]{LiuRao05})
\eq{
&x=\pmat{v\\v_t}, \quad A=\pmat{0&I\\\Delta&-a(z)}, \quad \Dom(A)=\Setm{(x_1,x_2)}{x_2\in H_0^1(\Omega), ~ \Delta x_1 \in \Lp[2](\Omega)}.
}
With these choices~\eqref{eq:waveex1} 
without the term $B_0C_0w(\cdot,t)$ on the right-hand side
becomes
\eq{
\dot{x}=Ax, \qquad x(0)=x_0,
}
where $x_0=(v_0,v_1)^T$.
Since $a(z)$ is strictly positive on a vertical strip of $\Omega$, we have from~\cite[Sec. 3, Ex. 3]{LiuRao05} that $A$ generates a strongly stable semigroup and $\norm{R(i\gw,A)}=\Omi(\abs{\gw}^2)$, and thus by Lemma~\ref{lem:BorTomresolvent} the semigroup generated by $A$ is polynomially stable with $\ga=2$.

In the composite system~\eqref{eq:waveextrisys} we choose the first subsystem as $X_1 = H_0^1(\Omega)\times \Lp[2](\Omega)$ and $A_1 = A$. The semigroup $T_1(t)$ generated by $A_1$ is then polynomially stable with $\ga_1=2$.

\subsection{The One-Dimensional System}
\label{sec:wave1D}

Now we turn our attention to the one-dimensional equation~\eqref{eq:waveex2} with the boundary and initial conditions in~\eqref{eq:wavebc}.
We define $g_0(r)=1-r$ and $A_0:\Dom(A_0)\subset \Lp[2](0,1)\rightarrow \Lp[2](0,1)$ as $A_0=-\frac{d^2}{dr^2}$ with the domain 
\eq{
\Dom(A_0)=\left\{x\in\Lp[2](0,1)~\right|&~x,x' ~ \text{abs. cont.,}  \,
\left. x''\in\Lp[2](0,1),~ x(0)=x(1)=0\right\}.
}
The operator $A_0$ has a positive self-adjoint square root 
\eq{
A_0^{1/2} x = \sum_{k=1}^\infty k\pi \iprod{x(\cdot)}{\sqrt{2}\sin(k\pi \cdot)}_{\Lp[2]} \sqrt{2}\sin(k\pi \cdot)
}
and the space $X=\Dom(A_0^{1/2})\times \Lp[2](0,1)$ is a Hilbert space with the inner product $\iprod{x}{y}_X= \iprod{A_0^{1/2} x_1}{A_0^{1/2} y_1}_{\Lp[2]}+\iprod{x_2}{y_2}_{\Lp[2]}$ for $x=(x_1,x_2)^T$ and $y=(y_1,y_2)^T$. Choosing
\eq{
&x=\pmat{w\\w_r}, \quad A=\pmat{0&I\\-A_0&0}, \quad \Dom(A)=\Dom(A_0)\times\Dom(A_0^{1/2}), \\[2ex]
& G u=g u=\pmat{0\\g_0}u, \quad
x_0=\pmat{w_0\\w_1},
}
the wave equation \eqref{eq:waveex2} can be written as
\eqn{
\label{eq:wave1Dlinsys}
\dot{x}=Ax+G u, \qquad x(0)=x_0.
}
We will now show that we can choose $K=\iprod{\cdot}{h}\in \Lin(X,\C)$ in such a way that with feedback input $u(t) = Kx(t)$ the system~\eqref{eq:wave1Dlinsys} is polynomially stable with $\ga=5/3$.
The eigenvalues of the operator $A$ are $\gl_k=ik\pi$ for $k\in\Z\setminus\set{0}$, and the corresponding eigenvectors 
\eq{
\varphi_k(z)=\half[\gl_k]\pmat{\sin(k\pi z)\\\gl_k\sin(k\pi z)}
}
form an orthonormal basis of $X$ and 
\eq{
\iprod{g}{\varphi_k}_X=\iprod{g_0}{\sin(k\pi\cdot)}_{\Lp[2]} =\frac{1}{k\pi}.
}
For $k\neq 0$ denote $\mu_k=-\frac{1}{\abs{k}^{5/3}}+ik\pi$. Then for any $\gl\in\C$ with $\dist(\gl,\gs(A))\geq \frac{\pi}{3} = \frac{1}{3}\inf_{k\neq l}\abs{\gl_k-\gl_l}$ and for any $l\neq 0$ we have
\begin{subequations}
\label{eq:poleplaceconds}
\eqn{
&\sum_{k\neq 0}\frac{\abs{\iprod{g}{\varphi_k}}^2}{\abs{\gl-\gl_k}^2}\leq \frac{1}{\pi^2\dist(\gl,\gs(A))^2}\sum_{k\neq 0}\frac{1}{k^2} <\infty \\
&\sum_{\substack{k\neq 0\\k\neq l}}\frac{\abs{\iprod{g}{\varphi_k}}^2}{\abs{\gl_k-\gl_l}^2}\leq \frac{1}{\pi^2} \sum_{\substack{k\neq 0\\k\neq l}} \frac{1}{k^2 \pi^2} <\infty\\
&\sum_{k\neq 0}\frac{\abs{\mu_k-\gl_k}^2}{\abs{\iprod{g}{\varphi_k}}^2}=\sum_{k\neq 0} \frac{\frac{1}{\abs{k}^{10/3}}}{\frac{1}{\pi^2k^2}} =\pi^2\sum_{k\neq 0} \frac{1}{\abs{k}^{4/3}}<\infty.
}
\end{subequations}
We now have from~\cite[Thm. 1]{xusallet} that there exists $h\in X$ such that $A+GK$ with $K=\iprod{\cdot}{h}_X$ is a Riesz-spectral operator with eigenvalues $\set{\mu_k}_{k\neq 0}$ and $A+GK$ has at most finite number of nonsimple eigenvalues. This immediately implies that for some constant $M\geq 1$ and for $\gw\in \R$ we have
\eq{
\norm{R(i\gw,A+GK)} \leq \frac{M}{\inf_k \dist(i\gw,\mu_k)} = \Omi(\abs{\gw}^{5/3}),
}
and thus the semigroup generated by $A+GK$ is polynomially stable with $\ga=5/3$.

In the composite system~\eqref{eq:waveextrisys} we choose $X_2 = \Dom(A_0^{1/2})\times \Lp[2](0,1)$ and $A_2 = A+GK$. We then have that the semigroup $T_2(t)$ generated by $A_2$ is polynomially stable with $\ga_2=5/3$.

\subsection{The Composite System}

If the space $X_1$ and $X_2$ and the operators $A_1$ and $A_2$ are chosen as in Sections~\ref{sec:wave2D} and~\ref{sec:wave1D}, then the coupled wave equations~\eqref{eq:waveex} can be written as a triangular system
\eq{
\ddb{t} \pmat{x_1\\x_2} = \pmat{A_1&BC\\0&A_2} \pmat{x_1\\x_2},
}
where the operators $B\in \Lin(Y,X_1)$ and $C\in \Lin(X_2,Y)$ are such that $By = (0,B_0y)^T \in X_1$ for $y\in Y$ and $C(x_2^1,x_2^2)= C_0 x_2^1$ for $x_2=(x_2^1,x_2^2)^T\in X_2$.

We can now use Theorem~\ref{thm:polpoltri} to pose conditions on $B$ and $C$ so that the triangular block operator matrix generates a polynomially stable semigroup.  Indeed, if these operators are such that $\ran(B)\subset \Dom(A_1)$ and $C(-A_2): \Dom(A_2)\rightarrow X_2$ extends to a bounded operator on $X_2$
(i.e., if $\gb=\gg=1$), then $\gb/\ga_1+\gg/\ga_2 = 1/2 + 3/5 = 11/10>1$, and Theorem~\ref{thm:polpoltri} concludes that the system~\eqref{eq:waveex} is polynomially stable with $\ga = \max \set{\ga_1,\ga_2}=2$. In particular, the space $Y$ does not have to be finite-dimensional. As an example, 
we can consider an interconnection of the form
\eqn{
\label{eq:wavecoupterm}
(B_0C_0 w(\cdot,t))(z) &= \sum_{k\neq 0} \frac{1}{k^2} \iprod{w(\cdot,t)}{\sin(k\pi \cdot)}_{\Lp[2]}  \sin(k z_1)\sin(k z_2) 
}
for $z=(z_1,z_2)\in \Omega$.
Here we can choose 
$ Y = \lp[2](\C) $
with the Euclidean basis vectors $\set{e_k}_{k\in\Z\setminus \set{0}}$, and define $B_0\in \Lin(Y,\Lp[2](\Omega))$ and $C_0\in \Lin( \Dom(A_0^{1/2}),Y)$ such that
\eq{
(B_0y)(z) &= \sum_{k\neq 0} \frac{1}{k^2} \iprod{y}{e_k} \sin(k z_1)\sin(k z_2)\\
C_0x_2^1 &= \sum_{k\neq 0} \iprod{x_2^1}{\sin(k\pi \cdot)}_{\Lp[2]} e_k
}
for $z=(z_1,z_2)\in \Omega$. We then have
\eq{
\norm{C_0x_2^1}^2 
&= \sum_{k\neq 0} \abs{\iprod{x_2^1}{\sin(k\pi \cdot)}_{\Lp[2]}}^2 
= 2\sum_{k=1}^\infty \abs{\iprod{x_2^1}{\sin(k\pi \cdot)}_{\Lp[2]}}^2\\
&\leq \sum_{k=1}^\infty k^2 \pi^2\abs{\iprod{x_2^1}{\sqrt{2}\sin(k\pi \cdot)}_{\Lp[2]}}^2
= \norm{x_2^1}_{\Dom(A_0^{1/2})}^2
}
For every $(x^1,x^2)^T\in \Dom(A_2) = \Dom(A_0)\times \Dom(A_0^{1/2})$ we also have that
\eq{
C(-A_2) \pmat{x^1\\ x^2}
=- C \pmat{0&I\\-A_0&0} \pmat{x^1\\ x^2}
- C GK \pmat{x^1\\ x^2}
= -C_0 x^2 
- C GK \pmat{x^1\\ x^2}.
}
Since $CGK\in \Lin(X_2,Y)$ and since
\eq{
\norm{C_0x^2}^2
&= \sum_{k\neq 0} \abs{\iprod{x^2}{\sin(k\pi \cdot)}_{\Lp[2]}}^2
= \sum_{k=1}^\infty \abs{\iprod{x^2}{\sqrt{2}\sin(k\pi \cdot)}_{\Lp[2]}}^2
= \norm{x^2}_{\Lp[2](0,1)}^2 \\
&\leq \norm{x^1}_{\Dom(A_0^{1/2})}^2
+ \norm{x^2}_{\Lp[2](0,1)}^2 
= \Norm{\pmat{x^1\\x^2}}^2,
}
we have that $C(-A_2)$ extends to a bounded operator on $X_2$, and thus $\gg=1$.

In order to verify that $B$ satisfies $\ran(B)\subset \Dom( (-A_1)^\gb)$ for $\gb=1$, we need to show that
$\ran(B_0)\subset H_0^1(\Omega)$.
Let $y\in Y=\lp[2](\C)$ be arbitrary. We have $(\iprod{y}{e_k}/k)_{k\neq 0}\in \lp[1](\C)$, and therefore
\eq{
\pdb{z_1}(B_0y)(z) 
&= \sum_{k\neq 0} \frac{1}{k^2} \iprod{y}{e_k} k\sin(k z_1)\cos(k z_2)
= \sum_{k\neq 0} \frac{1}{k} \iprod{y}{e_k} \sin(k z_1)\cos(k z_2)\\ 
\pdb{z_2}(B_0y)(z) 
&= \sum_{k\neq 0} \frac{1}{k} \iprod{y}{e_k} \sin(k z_1)\cos(k z_2).
}
Moreover, the property $(\iprod{y}{e_k}/k)_{k\neq 0}\in \lp[1](\C)$ implies that $\pdb{z_1}(B_0y)(\cdot)$ and $\pdb{z_2}(B_0y)(\cdot)$ are bounded uniformly continuous functions on $\overline{\Omega} = [0,\pi]\times [0,\pi]$. Since we also clearly have $(B_0y)(z)=0$ for every $z\in \partial\Omega$, this concludes $B_0y\in H_0^1(\Omega)$. The element $y\in Y$ was arbitrary, and we have thus shown that $\ran(B_0)\subset H_0^1(\Omega)$.

Since the conditions of Theorem~\ref{thm:polpoltri} are satisfied, we have that the system~\eqref{eq:waveex} of wave equations with the coupling~\eqref{eq:wavecoupterm} is polynomially stable with $\ga=\max \set{\ga_1,\ga_2}=2$.

\section{Conclusions}
\label{sec:conclusions}

In this paper we have studied the stability properties of a semigroup $T(t)$ generated by a $2\times 2$ block operator matrix $A$ under the assumption that the semigroups $T_1(t)$ and $T_2(t)$ generated by the diagonal operator blocks are polynomially stable. As our main results we presented conditions that guarantee polynomial stability of $T(t)$. We separately studied the situation where $A$ is a triangular block operator matrix. For such semigroups the conditions for polynomial stability are considerably weaker. In particular, for a full operator matrix the stability requires a condition that the graph norms
\eq{
  \norm{(-A_1)^{\gb_1}B_1} , \quad
  \norm{(-A_1^\ast)^{\gg_1}C_1^\ast} , \quad
  \norm{(-A_2)^{\gb_2}B_2} , \quad \mbox{and}\quad
  \norm{(-A_2^\ast)^{\gg_2}C_2^\ast} 
}
are sufficiently small. For a triangular $A$ no such condition is necessary. We also saw that in the case where one of the semigroups $T_1(t)$ and $T_2(t)$ is exponentially stable, the semigroup generated by a triangular $A$ is polynomially stable without any additional assumptions. In the case of a full operator matrix $A$, the stability of $T(t)$ still requires additional conditions on the exponents and the graph norms, as was illustrated in Example~\ref{ex:exppol}.

\end{document}